\documentclass{article} 

\usepackage{amsmath, amssymb, amsthm}
\usepackage{textcomp}
\usepackage{color}
\usepackage[titletoc, toc, title]{appendix}
\usepackage{hyperref}

\begin{document}
\newtheorem{theorem}{Theorem}[section]
\newtheorem{defin}[theorem]{Definition}
\newtheorem{remark}[theorem]{Remark}
\newtheorem{proposition}[theorem]{Proposition}
\newtheorem{lemma}[theorem]{Lemma}
\newtheorem{corollary}[theorem]{Corollary}
\def\div{\nabla\cdot}
\def\rot{\nabla\times}
\def\sign{{\rm sign}}
\def\arsinh{{\rm arsinh}}
\def\arcosh{{\rm arcosh}}
\def\diag{{\rm diag}}
\def\const{{\rm const}}
\def\eps{\varepsilon}
\def\phi{\varphi}
\def\theta{\vartheta}
\newcommand{\Bchi}{\mbox{$\hspace{0.12em}\shortmid\hspace{-0.62em}\chi$}}
\def\C{\hbox{\rlap{\kern.24em\raise.1ex\hbox
      {\vrule height1.3ex width.9pt}}C}}
\def\R{\mathbb{R}}
\def\P{\hbox{\rlap{I}\kern.16em P}}
\def\Q{\hbox{\rlap{\kern.24em\raise.1ex\hbox
      {\vrule height1.3ex width.9pt}}Q}}
\def\M{\hbox{\rlap{I}\kern.16em\rlap{I}M}}
\def\N{\hbox{\rlap{I}\kern.16em\rlap{I}N}}
\def\Z{\hbox{\rlap{Z}\kern.20em Z}}
\def\K{\mathcal{K}}
\def\A{\mathcal{A}}
\def\T{\mathcal{T}}
\def\D{\mathcal{D}}
\def\L{\mathcal {L}}
\def\({\begin{eqnarray}}
\def\){\end{eqnarray}}
\def\[{\begin{eqnarray*}}
\def\]{\end{eqnarray*}}
\def\part#1#2{{\partial #1\over\partial #2}}
\def\partk#1#2#3{{\partial^#3 #1\over\partial #2^#3}}
\def\mat#1{{D #1\over Dt}}
\def\dx{\, \nabla_x}
\def\dv{\, \nabla_v}
\def\dt{\, \partial_t}
\def\as{_*}
\def\d{\, \text{d}}
\def\e{\, \text{e}}
\def\pmb#1{\setbox0=\hbox{$#1$}
  \kern-.025em\copy0\kern-\wd0
  \kern-.05em\copy0\kern-\wd0
  \kern-.025em\raise.0433em\box0 }
\def\bar{\overline}
\def\lbar{\underline}
\def\fref#1{(\ref{#1})}

\begin{center}
{\LARGE 
Asymptotic analysis of a Vlasov-Boltzmann equation with anomalous scaling
}
\bigskip

{\large P. Aceves-S\'anchez}\footnote{Faculty of Mathematics, University of Vienna, Oskar-Morgenstern-Platz 1, 1090 Vienna, Austria.}
{\large \& A. Mellet}\footnote{Department of Mathematics, University of Maryland, College Park MD 20742.\hfill\break\indent\indent  Fondation Sciences Math\'ematiques de Paris, 75231 PARIS} 
\end{center}
\vskip 1cm

\noindent{\bf Abstract.} 
This paper is devoted to the approximation of the linear Boltzmann equation by fractional diffusion equations.
Most existing results address this question when there is no external acceleration field.
The goal of this paper is to investigate the case where a given acceleration field is present.
The main result of this paper shows that for an appropriate scaling of the acceleration field, the usual fractional diffusion equation is supplemented by an advection term.
Both the critical and supercritical case are considered.

\vskip 1cm


\medskip

\vskip 1cm

\noindent{\bf Acknowledgment:} A.M. was partially supported by NSF grant DMS-1501067. P.A.S gratefully acknowledges support from
the Consejo Nacional de Ciencia y Tecnologia of Mexico,
the PhD program \emph{Dissipation and
Dispersion in Nonlinear PDEs} funded by the Austrian Science Fund, grant no. W1245, and the Vienna Science and Technology Fund, grant no. LS13-029. 
P.A.S. acknowledges fruitful discussions with Christian Schmeiser.
The authors would also like to thank the Fondation Sciences Math\'ematiques de Paris, where this work was initiated, for its support.



\tableofcontents


\section{Introduction and main results}
\subsection{Introduction}
The goal of this paper is to study the asymptotic behavior of the solution of the following equation as $\varepsilon$ tends to zero:
\begin{equation}\label{kineticeq}
\left\{
\begin{array}{rcll}
\displaystyle \eps^{\alpha-1}  \partial_t f_\varepsilon + v \cdot \nabla_x f_\varepsilon + \frac{1}{\varepsilon^{2- \alpha}} E \cdot \nabla_v f_\varepsilon & = & \displaystyle \frac{1}{\eps}Q ( f_\varepsilon) & \text{in } ( 0 , \infty) \times \R^d \times \R^d, \\[3pt]
\displaystyle f_\varepsilon ( \cdot, \cdot, 0)  & = & \displaystyle f^{in}             & \text{in } \R^d \times \R^d, 
\end{array} \right.
\end{equation}
where $E \in\left[ W^{ 1, \infty} ( \R^d \times [ 0, \infty))\right]^d$ is a given acceleration field and 
$Q$ is the linear Boltzmann operator defined as
\begin{equation}\label{GeColOpe}
 Q ( f) : = \int \sigma ( v, v') M( v) f ( v') - \sigma ( v', v) M( v') f ( v) \d v' .
\end{equation}

Typically, $f_\eps(x,v,t)$ denotes the distribution function of some particles in a dilute gas, subject to an external acceleration field $E(x,t)$.
The small parameter $\eps$ can be interpreted as the Knudsen number, which measures the relative importance of the scattering phenomenon (described here by the collision operator $Q$) compared to the transport of particles ($\eps$ is often introduced in the literature as
the ratio of the mean free path over some typical macroscopic length, such as the length of the device being studied).
The coefficient $\alpha$ determines the relative order of the various terms in \eqref{kineticeq} and it will be fixed by the properties of the  thermodynamical equilibrium $M(v)$ appearing in the operator $Q$. 
One possible definition for $\alpha$ is
\begin{equation}\label{def:alpha} 
\alpha = \sup \left\{\beta\leq 2 \, ;\, \int_{\R^d} |v|^\beta M(v)\, dv <\infty\right\}.
\end{equation}
However, we will make stronger assumptions on the behavior of $M$ for large $|v|$ which will make the definition of $\alpha$ simpler.
Concerning the particular choice of scaling in \eqref{kineticeq}, we note that  the $\eps^{\alpha-1}  $ 
in front of the time derivative corresponds to a particular choice of a time scale at which we know that diffusion will be observed (\cite{MR2763032,MR2815035}), while the $\frac{1}{\varepsilon^{2- \alpha}}$ in front of the force term correspond to a strong field assumption (we will always have $\alpha<2$ and so $\frac{1}{\varepsilon^{2- \alpha}}\gg 1$). Obviously other choices of scaling for this force term are possible (see Remark \ref{rem:hf}), but this particular scaling is exactly the one for which the diffusion process (due to the scattering  phenomenon of $Q$) and the advection process (due to the acceleration term $E$) are of the same order in the limit (see equations \eqref{drifdifeq} and \eqref{drifdifeq2}).

\medskip

When $M(v)$ is a Maxwellian distribution function, or more generally when $M(v)$ satisfies
$$\int_{\R^d} |v|^2 M(v)\, dv <\infty,$$ then \eqref{def:alpha} gives $\alpha=2$ and,
we recognize in \eqref{kineticeq} the classical drift-diffusion scaling. 
If we assume further that $E=0$, then such limits were first investigated in 
the pioneering works \cite{Habetler1975}, \cite{Bensoussan1979}, \cite{Wigner1961} 
and \cite{Larsen1974}. In all these papers, it is assumed that $M$ is a Maxwellian distribution function;
In \cite{MR1803225}, Degond-Goudon-Poupaud extended these results to a more general distribution $M$, but always under the assumption of finite second moment.
The case $E\neq 0$ is addressed for example by Poupaud in \cite{poupaud1991diffusion}
when $M$ is a Maxwellian. It is shown in particular that the addition of the force field $E$ leads to a drift term in the limiting equation for the density of particles.

\medskip

The object of this paper is to investigate 
what happens when $M(v)$ has a so-called heavy tail distribution function with $\alpha<2$.
To be more precise, we will assume that 
$$ M(v)\sim \frac{\gamma}{ |v|^{d+\alpha} } \mbox{ as } |v|\to\infty$$
for some $\alpha<2$. The $\alpha$ describing the large velocity behavior of $M(v)$ is then the same as the $\alpha$ appearing in \eqref{kineticeq} (this is consistent with \eqref{def:alpha}).
When $E=0$, such limits  have been the object of several recent works  
(see for example \cite{MR2763032}, \cite{MR2815035}, and \cite{ben2011anomalous}),
and it  has been shown that the limiting behavior of $f_\eps$ is described by a fractional diffusion equation.

The main contribution of the paper is thus to consider the case $E\neq 0$.
In view of the scaling in equation \eqref{kineticeq}, we immediately note that the cases $\alpha\in(1,2)$, $\alpha=1$ and $\alpha\in(0,1)$ are radically different.
Indeed, when $\alpha\in(1,2)$, all the terms in the left hand side of  \eqref{kineticeq} are smaller than $\eps^{-1}$ when $\eps\ll 1$. So, assuming that $f_\eps$ converges to $f$ (for instance in $\D'$), we immediately get 
$ Q(f)=0$, that is 
$$ \lim_{\eps\to 0} f_\eps(x,v,t) = \rho(x,t) M(v).$$
By contrast, when $\alpha=1$, the force term is of the same order as the collision term, and we will get instead
$$ \lim_{\eps\to 0} f_\eps(x,v,t) = \rho(x,t) F(x,v,t)$$
where $F$ is the unique  solution of
\begin{equation}\label{eq:kernel0}
 Q(F)-E\cdot\nabla_vF = 0, \qquad \int_{\R^d} F\, dv =1,
 \end{equation}
(see Proposition \ref{pro:Fexis} below for the existence of $F$). Equation \eqref{eq:kernel0} classically appears in the high field asymptotic limit which has been studied for various operators $Q$ \cite{Arlotti,Po1992,BenChaker} (see also Remark \ref{rem:hf} below).
Finally, when $\alpha\in(0,1)$, the force term in the left hand side of \eqref{kineticeq} is more singular than the collision term, and the limit $f(x,v,t)$ of $f_\eps(x,v,t)$  satisfies
$$ E\cdot \nabla_v f = 0.$$
It is not clear to us what one could expect to prove in this last case. In fact, we will see that we are not able to obtain a priori estimates on $f_\eps$ to successfully  investigate such a limit (note however that $f_\eps$ is always bounded in $L^\infty(0,\infty; L^1(\R^d\times\R^d))$, so some limit always exists).
In this paper, we thus focus our attention on the two cases $\alpha\in(1,2)$ and $\alpha=1$. One of the key observations that allowed us to 
obtain the hydrodynamic limit in a rigorous manner is to note that not only the operator $Q$ appearing in \eqref{kineticeq} is coercive but also the operator 

\[
 \T ( f) := - Q ( f) + E \cdot \nabla_v f
\]
is coercive in a suitable space (see Proposition \ref{prop:coer}). Our proof is based on analytic methods.

We will show that the limit $f$ of $f_\eps$ is of the form $\rho(x,t) M(v)$ (or 
$\rho(x,t) F(x,v,t)$ when $\alpha=1$) where $\rho$ solves a fractional diffusion equation of order $\alpha$ with a drift term.
In that spirit, the first derivation of a fractional diffusion equation with an advection term 
starting from a kinetic model
was
first obtained in \cite{AceSch} as a macroscopic limit of an equation featuring a collision operator with
a biased velocity.
Note that 
evolution equations involving a fractional-diffusion term appear in many equations of mathematical physics 
(consult \cite{Vazquez2014857} and \cite{Shlesinger1995}, and the references therein),
for instance in fluid dynamics with the so-called quasi-geostrophic
flow model (see \cite{Constantin2006}) (in that case the equation is 
non linear  since the drift depends on the solution). 
The study of fractional-diffusion advection equations has been a very active field of research recently, and questions such as the regularity of
the solutions have been addressed, see for instance \cite{Silvestre2012} and \cite{SilVicolZlatos}.
It is a classical fact that the case of the half Laplacian ($\alpha=1$ with our notations) plays a critical role in that case since the diffusion operator has the same order as the advection term.
In that sense, it is not surprising that the case $\alpha=1$ plays a critical role in our study as well.

\medskip

\subsection{Assumptions}

We now list our main assumptions.
As noted above, the acceleration field $E(x,t)$ is assumed to be given (as opposed to, say, solution of Poisson equation), and satisfies
\begin{equation}\label{eq:E}
 E\in \left[W^{1,\infty} \left(\R^d \times (0,\infty) \right)\right]^d.
 \end{equation}
 
 Next,  we assume that $M$ satisfies:
\begin{align}
 &M>0,\quad M( v) = M( -v) \mbox{ for all } v \in \R^d, \quad \int_{\R^d} M(v) \d v = 1,\label{eqprop}\\
 & |v|^{d+\alpha} M(v) \longrightarrow \gamma>0  \,,\qquad\mbox{as } |v| \rightarrow \infty ,\quad \mbox{where } 1\leq \alpha<2,  \label{eqdecay} 
\end{align}
as well as the  following regularity assumptions:
\begin{equation} 
|D_v M(v)|\leq C\frac{M(v)}{1+|v|}, \quad 
|D_v^2 M(v) | \leq C M(v) . \label{derdecay2}  
\end{equation}
We note that these assumptions are compatible with the asymptotic behavior of $M$ given by \eqref{eqdecay}.
They are in particular satisfied by the function
$$ M(v)=\left(\frac{1}{1+|v|^2}\right)^{\frac{ d + \alpha}{2}}$$
and by the probability density function of the so-called $\alpha$-stable stochastic processes \cite{Bogdan+2007}.

\medskip

The cross section $\sigma(v,v')$ appearing in the operator $Q$ will be assumed to satisfy

\begin{align}
& \qquad 
 \sigma ( v, v') = \sigma ( v', v), \qquad   \nu_1 \leq \sigma ( v, v') \leq \nu_2,\quad \text{for all } v, v' \in \R^d\label{sigmaProp}\\
& \hspace{3cm} | \nabla_v \sigma ( v', v) | \leq \frac{ C}{ 1 + | v|}, \label{nuprop}
\end{align}
where $C$, $\nu_1$ and  $\nu_2$ are positive constants. Let us note that the symmetry condition \eqref{sigmaProp} on 
$\sigma$ guarantees that $Q(M)=0$. If we define the collision frequency $\nu(v)$ by  
\[
 \nu ( v) = \int \sigma ( v', v) M( v') \d v'
\]
then conditions \eqref{sigmaProp} and \eqref{nuprop} imply
\begin{equation}
\nu_1\leq\nu(v)\leq\nu_2,
\qquad  \left| \nabla_v \nu(v) \right| \leq \frac{C}{1+|v|}  \qquad \mbox{ for all } v\in \R^d.
  \label{nuprop2}
\end{equation}
In addition, we assume that the collision frequency $\nu$ is even, namely,

\begin{equation}\label{nueven}
 \nu ( v) = \nu ( -v) \qquad \mbox{ for all } v \in \R^d.
\end{equation}

Finally, we need $\sigma$ and $\nu$ to have a nice behavior as $v\to\infty$. More precisely, we assume:
\begin{equation}\label{eq:silim}
|\sigma(v,v') - \nu_0| \leq \frac{C}{1+|v|} \quad \mbox{ for all $v,v'\in \R^{2d}$ },
\end{equation} 
for some $\nu_0$,
which implies in particular 
  \begin{equation}\label{eq:nuasym}
 \nu ( v) \to \nu_0, \text{ as } | v| \to \infty. 
  \end{equation}
  
\medskip

\subsection{Main results}

Under assumptions \eqref{eq:E} and \eqref{sigmaProp},
the existence and uniqueness of a solution $f_\eps\in \mathcal{C}^0 ( [ 0, \infty); L^1 ( \R^d \times \R^d))$  to  \eqref{kineticeq}
can be proved via a semigroup argument. 
We do not discuss this issue here and refer instead the interested reader to \cite{poupaud1991diffusion} or \cite{MR1295030} for the existence
of a mild solution and to the Appendix of \cite{DiLi1989} where the equivalence between the mild solution and a solution in the 
sense of distributions is shown.

\medskip

In this paper  we investigate the asymptotic behavior of $f_\eps$ as $\eps\to 0$.
Our first result concerns the case $\alpha\in(1,2)$:
\begin{theorem}\label{thm:main}
Assume $\alpha\in(1,2)$ and let $f_\eps ( x, v, t)$ be the solution of \eqref{kineticeq} with initial condition 
 $f^{in}\geq 0$ satisfying 
 $$
f^{in}  \in L^2_{ M^{ -1}} ( \R^d \times \R^d) \cap L^1 ( \R^d \times \R^d).$$
Under Assumptions \eqref{eq:E}-\eqref{eq:silim} listed above,
the function $f_\eps(x,v,t)$ converges weakly  in  $\star$-$L^\infty ( 0, T; L^2_{ M^{ -1}} ( \R^d \times \R^d))$
to the function $\rho(x,t) M(v)$
where $\rho$ solves 
\begin{equation}\label{drifdifeq}
\left\{
\begin{array}{rcll}
\partial_t \rho + \kappa ( - \Delta)^{ \alpha/2} \rho + \nabla_x \cdot ( D E \rho ) & = & 0 & \text{in } ( 0 , \infty) \times \R^d , \\
\rho ( \cdot, 0)  & = & \rho^{in}             & \text{in } \R^d,
\end{array} \right.
\end{equation}
with $\rho^{ in} ( x) = \int f^{ in} ( x, v) \d v$ and with the
coefficient $\kappa$ and 
 matrix $D$  defined by 
 \begin{equation}\label{kapval} 
 \kappa =  \frac{ \gamma\nu_0^2}{ c_{ d, \alpha}} \int_0^\infty z^\alpha \e^{ -\nu_0z} \d z,
 \end{equation}
 and 
 \begin{equation}\label{eq:D}
D = \int \lambda(v)\otimes v\, dv, \qquad Q(\lambda)=\nabla_v M(v).
\end{equation}
\end{theorem}
Note that the constant $c_{ d, \alpha}$ appearing in \eqref{kapval}  is defined in \eqref{eq:fracla} and  that the existence of the function $\lambda(v)$ appearing in \eqref{eq:D} will be proved in Lemma \ref{lem:lambda0}.
When $\sigma(v,v')=1$, we can take $\lambda(v) = -\nabla_v M(v)$, and we can check that $D$ is the identity matrix.

\medskip

Next, we consider the critical case $\alpha=1$.
In that case, Equation \eqref{kineticeq} reads
$$
\left\{
\begin{array}{rcll}
 \partial_t f_\varepsilon + v \cdot \nabla_x f_\varepsilon + 
\frac{1}{\eps} E \cdot \nabla_v f_\varepsilon & = & \frac{1}{\eps}Q ( f_\varepsilon) & \text{in } ( 0 , \infty) \times \R^d \times \R^d, \\
f_\varepsilon ( \cdot, \cdot, 0)  & = & f^{in}             & \text{in } \R^d \times \R^d, 
\end{array} \right.
$$
and we recognize the so-called high field asymptotic for the Boltzmann equation.
Such asymptotics were first studied by Arlotti and Frosali \cite{Arlotti} and Poupaud  \cite{Po1992}
for the linear Boltzmann operator with Maxwellian equilibrium 
(see also Ben Aballah-Chaker \cite{BenChaker} for a non-linear collision operator).
The main difference in this case is that the weak limit of $f_\eps$ 
will be the solution $F$ of \eqref{eq:kernel0} (which depends on $E$) rather than $M(v)$.
The existence and properties of $F$ will be the object of Theorem \ref{thm:FE}
below.
In particular, we will prove that there exists a function $F(v,E)$ defined for $(v,E)\in \R^d\times\R^d$ such that for all $E\in\R^d$, $v\mapsto F(v,E)$ solves
\begin{equation} \label{eq:FFEE} Q(F)-E\cdot\nabla_vF = 0, \qquad \int_{\R^d} F(v,E)\, dv =1.
\end{equation}
We then have:
\begin{theorem}\label{thm:2}
Assume $\alpha=1$ and let $f_\eps ( x, v, t)$ be the solution of \eqref{kineticeq} with initial condition 
 $f^{in}\geq 0$ satisfying 
 $$
f^{in}  \in L^2_{ M^{ -1}} ( \R^d \times \R^d) \cap L^1 ( \R^d \times \R^d).$$
Under Assumptions \eqref{eq:E}-\eqref{eq:silim} listed above,
the solution $f_\eps(x,v,t)$ of \eqref{kineticeq} converges weakly  in  $\star$-$L^\infty ( 0, T; L^2_{ M^{ -1}} ( \R^d \times \R^d))$
to the function 
$\rho(x,t)  F(v,E(x,t))$
where $\rho(x,t)$ solves
\begin{equation}\label{drifdifeq2}
\left\{
\begin{array}{rcll}
\partial_t \rho + \kappa ( - \Delta)^{1/2} \rho + \mbox{div}_x (  \mu ( E) \rho ) & = & 0 & \text{in } ( 0 , \infty) \times \R^d \times \R^d, \\
\rho ( \cdot, 0)  & = & \rho^{in}             & \text{in } \R^d, 
\end{array} \right.
\end{equation}
where $\rho^{ in} ( x) = \int f^{ in} ( x, v) \d v$,
\begin{equation}\label{eq:mu}
 \mu( E) := \int v \left( F ( v, E) - M( v) \right) \d v,
\end{equation}
and
$$ \kappa =  \frac{ \gamma\nu_0^2}{ c_{ d, 1}} \int_0^\infty z \e^{ -\nu_0z} \d z.
$$
\end{theorem}

This result should be compared to the classical high-field limit (\cite{Arlotti,Po1992}), which leads to a transport equation.
Here the (fractional) diffusion takes place at the same time scale as the transport and thus appears in the limiting equation.

\medskip

Note that the fact that $\mu(E)$ is well defined  by formula \eqref{eq:mu} is not completely obvious since
$vM(v)$ is not integrable when $\alpha=1$. However, we will see in Lemma \ref{lem:mu} that $ F ( v, E) - M( v)$ decays faster than $M$ and that $\mu(E)$ is indeed well defined.

When $\sigma$ is constant, we can get explicit formulas for $F(v,E)$ and $E$.
Indeed, if $\sigma = 1$ then the operator $Q$ reads
$$ Q(f)(v) = \int_{\R^d } f(v')\, \d v'  M(v) - f(v)$$
and equation  \eqref{eq:kernel0} can be recast as
$$ F + E\cdot \nabla_v F  = M$$
which can be explicitly integrated along the characteristics yielding the following formula:
\begin{equation}\label{eq:FFE}
F(v,E)=\int_0^\infty \e^{ -z} M ( v - E z) \d z.
\end{equation}
We can also use the equation above to compute
$$ \mu(E)= - \int  v E\cdot \nabla_v F\,\d v = E 
 \qquad \text{for all } E\in\R^d$$
(using an integration by part and the fact that $\int F(v) \d v=1$).

\begin{remark}\label{rem:hf}
When $M$ satisfies \eqref{eqdecay} with
$\alpha\in(1,2)$, we can also consider the high field asymptotic regime as in \cite{Arlotti,Po1992}. 
It corresponds to the following scaling of the equation:
$$
\left\{
\begin{array}{rcll}
 \partial_t f_\varepsilon + v \cdot \nabla_x f_\varepsilon + 
\frac{1}{\eps} E \cdot \nabla_v f_\varepsilon & = & \frac{1}{\eps}Q ( f_\varepsilon) & \text{in } ( 0 , \infty) \times \R^d \times \R^d, \\
f_\varepsilon ( \cdot, \cdot, 0)  & = & f^{in}             & \text{in } \R^d \times \R^d, 
\end{array} \right.
$$
In that case, it is relatively easy to show that $f_\eps$ converges to
$\rho(x,t)  F(v,E(x,t))$
where $F$ is given by \eqref{eq:FFEE}
and $\rho$ solves the transport equation
$$ \dt \rho + \mbox{div}_x(\rho E) = 0 .$$
\end{remark}

\medskip

\begin{remark}
The case $\alpha=2$ is also interesting. In this case 
the scaling in equation \eqref{kineticeq} becomes the usual diffusion scaling, however, the second moment $\int |v|^2 M(v)\, dv $ (and thus the diffusion coefficient) is infinite.
This critical case was studied in \cite{MR2763032}, and it was shown that the time scale must be modified by a logarithmic factor, leading to the following equation:
$$ \eps  \ln(\eps^{-1})\partial_t f_\varepsilon + v \cdot \nabla_x f_\varepsilon + \ln(\eps^{-1})  E \cdot \nabla_v f_\varepsilon  =  \displaystyle \frac{1}{\eps}Q ( f_\varepsilon) .
$$
The limiting equation, on the other hand, will now involve the  regular Laplace operator.
\end{remark}

\subsection{Notations and organization of the paper}

We recall that the fractional Laplacian appearing in \eqref{drifdifeq} and \eqref{drifdifeq2} can be defined via the Fourier transform as
\begin{equation}\label{fracLaplacian}
\mathcal{F} \big( ( - \Delta)^{\alpha /2} f \big) ( k) : = |k|^\alpha \mathcal{ F} ( f) ( k),
\end{equation}

\noindent where $\mathcal{F} ( f)$ denotes the Fourier transform of $f$ and is defined as

\begin{equation}
 \mathcal{F} ( f) := \int e^{ - i k \cdot x} f ( x) \d x,
\end{equation}
or as a singular integral as

\begin{equation}\label{eq:fracla}
( - \Delta)^{ \alpha / 2} f( x) = c_{ d, \alpha} \, \, \text{P.V.} \int_{ \R^d} \frac{ f( x) - f( y)}{ | x - y|^{ d + \alpha}} \d y, 
\end{equation}

\noindent where P.V. denotes the Cauchy principal value and 

$$
c_{ d, \alpha} = \frac{\alpha 2^{ \alpha - 1} \varGamma ( \frac{ \alpha + N}{ 2})}{ \pi^{ N/2}\varGamma ( \frac{ 2 - \alpha}{ 2})}, 
$$

\noindent where $\varGamma ( x)$ is the Gamma function. 
When $\alpha > 1$,   the principal value can be avoided by 
using the following formula:
\[
 ( - \Delta)^{ \alpha / 2} f ( x) = c_{ d, \alpha} \int_{ \R^d} \frac{ f ( x) - f( y) - \nabla_x f ( x) ( x - y)}{ | x - y|^{ d + \alpha}} \d y \, .
\]
For a detailed discussion on the properties of the fractional Laplacian consult 
\cite{MR0350027}, \cite{stein1970singular}, or \cite{DiNezza2012521}.

\medskip


We denote by 
$\d x$, $\d v$ and $\d v'$ the Lebesgue measure on $\R^d$ and by $\d t$ the Lebesgue
measure on $[ 0, \infty )$, where $\R^d$ and $[ 0, \infty )$ will be the integration domains, respectively, unless stated otherwise.
We will denote by $L^2_{M^{-1}}(\R^d)$ (respectively $L^2_{F_\eps^{-1}}(\R^d)$) the space of square integrable function with weight $M^{-1}$ (respectively $F_\eps^{-1}$) equipped with the norm
$$ \| f\|_{L^2_{M^{-1}}(\R^d)} = \left( \int_{\R^d} |f(v)|^2 \frac{\d v}{M(v)}\right)^{1/2}.$$
\medskip

Finally, given a function $f \in L^1 \left( \R^d \right)$ we define the mass density $\rho_f$ of $f$ as
\begin{equation}\label{def:massden}
 \rho_f : = \int f \d v.
\end{equation}

\medskip

The rest of the paper is organized as follows:
In the next section, we prove the existence of $F$, solution of \eqref{eq:FFEE}, and we investigate its properties.
In Section 3, we will derive the a priori estimates on $f_\eps$ solution of \eqref{kineticeq} which will be necessary for the proofs of our main results.
Finally, Theorem \ref{thm:main} is proved in Section 4, while Theorem \ref{thm:2} is proved in Section 5.
\medskip

\section{The modified equilibrium function $F$}
Classically, a priori estimates for the solutions of \eqref{kineticeq}
 are obtained as consequence of the following coercivity property of the Boltzmann collision operator:
\begin{lemma}\label{lem:Q}
Under assumption \eqref{sigmaProp}, the operator $Q$ is a bounded operator in $L^2_{M^{-1}}(\R^d)$ 
which satisfies the following coercivity estimate:

$$ 
-\int_{\R^d} Q(f)f\frac{dv}{M(v)} \geq \nu_1 \int_{\R^d} |f - \rho_f M|^2\, \frac{\d v}{M(v)}, 
$$
for all $f\in L^2_{M^{-1}}(\R^d)$ and with $\rho_f$ given by \eqref{def:massden}.
\end{lemma}
When $E=0$, this very classical lemma immediately implies that the solution of \eqref{kineticeq} satisfies
$$ f_\eps(x,v,t) = \rho_\eps(x,t) M(v) + \eps^{\alpha/2}r_\eps(x,v,t)$$
where the remainder term $r_\eps$ is bounded in some appropriate functional space (such a bound is obtained by multiplying \eqref{kineticeq} by $f_\eps/M$ and integrating).
Such estimates can be generalized to include the case $E\neq 0$ and $\alpha=2$. Unfortunately, these computations do not seem to be useful in the case $\alpha<2$ which we are considering here.

In the next section,
we will see that we can instead obtain the following expansion for $f_\eps$:
$$ f_\eps(x,v,t) = \rho_\eps(x,t) F_\eps(x,v,t) + \eps^{\alpha/2}r_\eps(x,v,t)$$
where $F_\eps$ is the normalized equilibrium function solution of 
 \begin{equation}\label{kernel}
   \eps^{ \alpha - 1} E \cdot \nabla_v F_\eps - Q ( F_\eps) = 0 \, , \qquad \int F_\eps \d v = 1.
\end{equation}
Our  goal in this section is to prove the existence and uniqueness of $F_\eps$ and study its properties.

\medskip

But first we note that we can write 
$$F_\eps(x,v,t) = F(v,\eps^{\alpha-1} E(x,t))$$ 
where the function $v\mapsto F(v,E)$ solves (for all $E\in\R^d$):
\begin{equation}\label{eq:kernelF}
 E\cdot \nabla_v F - Q(F)=0 , \qquad \int_{\R^d} F(v,E)\, dv =1. 
\end{equation}
This equation plays a central role in the study of the  high field asymptotics for Boltzmann type equations, 
and has been studied for various operators $Q$.
However, it does not seem that it has been studied under our assumptions on the function $M(v)$ (property \eqref{pEF} below, in particular, is very specific to our framework).
We will thus study \eqref{eq:kernelF} in detail in this section.
More precisely, gathering all the key results that we will  prove in this section, we have the following:
\begin{theorem}\label{thm:FE}
\item[(i)] For all $E\in\R^d$, the exists a unique function $v\mapsto F(v,E)$ solution of \eqref{eq:kernelF}.
\item[(ii)] There exist two positive constants $C(R)$ and $c(R)$ such that if $|E|\leq R$ then
$$
 c(R) M(v)\leq F(v,E)\leq C(R) M(v) \qquad \text{for all } v\in \R^d.
$$
\item[(iii)] The function $E\mapsto F(v,E)$ is $C^1$ and for all $R>0$ there exists $C(R)$ such that
   \begin{equation}\label{pEF}
   | \partial_E F(v,E)| \leq C(R) \frac{F(v,E)}{1+|v|}\quad  \mbox{ for all $v\in\R^d$ and  $|E|\leq R$}.
   \end{equation}
\end{theorem}

Since we are assuming $\alpha\geq1 $, assumption \eqref{eq:E} implies that $|\eps^{\alpha-1} E(x,t)| $
is bounded uniformly in $\eps$, $x$ and $t$, and so 
the results of this theorem will apply to the function $F_\eps(x,v,t) = F(v,\eps^{\alpha-1} E(x,t))$ 
(see Propositions \ref{prop:FMbd} and  \ref{FpropGe}).
When $\alpha>1$, 
the behavior of $F(v,E)$ for $|E|\ll 1$ will play an important role.
We will thus prove the following result:
\begin{proposition}\label{prop:Fas}
The following expansion holds:
\begin{equation} \label{eq:FGE}
F (v,E) = M(v) + E\cdot \lambda(v) + G(v,E)
\end{equation}
where $\lambda(v)$ is such that 
$$ Q(\lambda)(v) =\nabla_v M(v) , \qquad \int_{\R^d} \lambda(v) \, dv =0,$$
and $G$ satisfies:
\begin{equation}\label{eq:GL22}
 \|G(\cdot,E)\|_{L^2_{M^{-1}} \left( \R^d \right) } \leq C  | E|^2 \quad \mbox{ for all  $|E|\leq 1$}
 \end{equation}
and 
\begin{equation}\label{eq:GLinfty2} 
| G(\cdot,E)| \leq  C  | E|^2 M(v) \quad \mbox{ for all $v\in\R^d$, $|E|\leq 1$.}
\end{equation}
\end{proposition}

\subsection{Existence of $F(v,E)$}
In this Section, we prove the existence of a unique solution to \eqref{eq:kernelF} (Theorem \ref{thm:FE} $(i)$). The proof  follows closely the arguments of Poupaud in \cite{Po1992}. We recall it here for the sake of completeness.
Throughout this section, we fix $E\in \R^d$ and we define the operator
\begin{equation}\label{eq:TE}
 \T ( f) : = - Q ( f) + E \cdot \nabla_vf .
 \end{equation}
 We also define the operators $\A$ and  $\K$ by 
  \[
   \A ( f) :=   E \cdot \nabla_v f + \nu f, \qquad \K ( f) := \int \sigma ( v, v') f( v') \d v' M( v)
  \] 
so that $\T= \A -\K$.
We note that $\K$ is a positive compact operator in 
  $ L^2_{ M^{ -1}} \big( \R^d \big) $ (it is a Hilbert-Schmidt operator),
  while $\A$ is an unbounded operator with domain
  \begin{align}
   D ( \A) & := \Big\{ f \in L^2_{ M^{ -1}} \big( \R^d \big) \, | \,   E \cdot \nabla_v f   \in L^2_{ M^{ -1}} \big( \R^d \big) \Big\}. \label{basDom} 
  \end{align}
Furthermore, we can define the inverse operator
$\A^{ -1} $ as follows:
   \begin{equation}\label{invfor}
   \A^{ -1}  ( h) : = \int_0^\infty \e^{ - \int_0^s \nu ( v -  E \tau) \d \tau}  h ( v -   E s) \d s.
  \end{equation}
Indeed, we have:
\begin{lemma}\label{Ainv}
The operator $\A^{ -1}$ defined by \eqref{invfor} is a bounded operator in $L^2_{ M^{ -1}}$
(with a norm depending on $|E|$)
 which satisfies 
$$( \A \circ \A^{ -1} ) ( f) = f \quad \mbox{ for all $f \in L^2_{ M^{ -1}} \left( \R^d \right)$},
$$
and
$$
( \A^{ -1} \circ \A ) ( f) = f \quad \mbox{ for all $f \in D ( \A)$}.
$$
\end{lemma}

Postponing the proof of this Lemma to the end of this section, 
we first show that it implies the main result of this section:

\begin{proposition}\label{pro:Fexis}
   For all $E \in\R^d$,  
   there exists a unique positive solution $v\mapsto F(v,E)$  of  \eqref{eq:kernelF}
   in  $L^2_{M^{-1}}(\R^d)$.
\end{proposition}

\begin{proof}[Proof of Proposition \ref{pro:Fexis}]
We can rewrite \eqref{eq:kernelF} as
   \begin{equation}\label{eqdist}
 \A F = \K ( F), \qquad  \int F \d v = 1.
  \end{equation}
Formula \eqref{invfor}  shows that $\A^{-1}$ is a nonnegative operator (if $h \geq 0$ then $\A^{ -1} ( h) \geq 0$).
It follows that the operator  $\K \circ \A^{ -1}$ is a positive compact operator
 in $L^2_{M^{-1}} ( \R^d)$ and so we can apply 
 Krein-Rutman's Theorem (see \cite{KrRu1950}) to deduce
the existence of a unique simple positive eigenvalue $\lambda$ with associated
  positive eigenfunction $W$ satisfying 
  $$\left( \K \circ \A^{ -1} \right) W= \lambda W.$$
We now define $F := \A^{-1} W$ and note that thanks to Lemma \ref{Ainv} it satisfies

$$
\K ( F) = \lambda \A F.
$$
Integrating this relation with respect to $v$ and using the definition of $\nu$, we find
\[
 \int \nu ( v) F ( v) \d v = \lambda\int \nu ( v) F ( v) \d v,
\]
from which it follows that $\lambda = 1$. After normalizing $F$ the proposition follows.
\end{proof}

We complete this section with a proof of Lemma \ref{Ainv}:

\begin{proof}[Proof of Lemma \ref{Ainv}]
The fact that $( \A \circ \A^{ -1}) ( f) = f$ for all $f \in L^2_{ M^{ -1}} \left( \R^d \right)$, and
$( \A^{ -1} \circ \A) ( f) = f$ for all $f \in D( \A)$ can be proved as the Proposition~1 in~\cite{Po1992}. 
   
To show that $\A^{-1}$ is a bounded operator, we first note (using \eqref{nuprop2}) that
$$ 
| \A^{ -1} ( h)|  \leq \int_0^\infty \e^{ -\nu_1 s}  h ( v -  E s) \d s.
$$
We thus have
\begin{align*}
\int_{\R^d} | \A^{ -1} ( h)|^2 \frac{\d v}{M(v)}  & \leq 
 \frac{1}{\nu_1} \int_{\R^d} \int_0^\infty \e^{ -\nu_1 s}  \frac{|h ( v -  E s)|^2}{M(v)} \d s \, dv \\
  & \leq 
 \frac{1}{\nu_1} \int_{\R^d} \int_0^\infty \e^{ -\nu_1 s}  \frac{|h ( v)|^2}{M(v +  E s)} \d s \, dv \\
 & \leq 
 \frac{1}{\nu_1} \int_{\R^d} \left( \int_0^\infty \e^{ -\nu_1 s} \frac{M(v)}{M(v +  E s)} \d s\right) \frac{|h (v )|^2}{M(v)} \, dv, 
\end{align*}   
and we conclude thanks to the following claim: There exists a $C > 0$ such that
\begin{equation*}
    \left( \int_0^\infty \e^{ -\nu_1 s} \frac{M(v)}{M(v +   E s)} \d s\right) \leq C(1+  |E|^{d+\alpha})) \quad \mbox{ for all $v\in \R^d$.}
\end{equation*}
This last bound is proved by first noticing that \eqref{eqdecay} implies, in particular, the existence of $\mu_1, \mu_2 > 0$ such that
 \begin{equation}\label{comparison}
  \frac{ \mu_1}{ 1 + | v|^{ d + \alpha}} \leq M( v) \leq \frac{ \mu_2}{ 1 + | v|^{ d + \alpha}} \quad \mbox{for all } v \in \R^d.
 \end{equation}
Therefore, using the elementary inequality $| a + b |^p \leq C \left( | a|^p + | b|^p \right)$, valid for $p \geq 1$, we obtain the following estimate:
\begin{align*}
\int_0^\infty \e^{ -\nu_1 s} \frac{M(v)}{M(v +  E s)} \d s
& \leq \frac{\mu_2}{\mu_1} \int_0^\infty \e^{ -\nu_1 s} \frac{1+|v +  E s|^{d+\alpha}}{1+|v|^{d+\alpha}} \d s \\
& \leq C \frac{\mu_2}{\mu_1} \int_0^\infty \e^{ -\nu_1 s}(1+|  E s|^{d+\alpha}) \d s \\
& \leq C (1+|  E |^{d+\alpha} ).
\end{align*}

\end{proof}

\subsection{Properties of $F(v,E)$: Theorem \ref{thm:FE} (ii)}
As noted in the Introduction, in the simpler case where the cross section satisfies
$$\sigma(v,v')=1 \quad\mbox{ for all } v,\,v'\in \R^d ,$$
the equation for $F$ reduces to
$$
  F + E \cdot \nabla_v F = M(v),
$$
and we  get the following explicit formula for $F$:
\begin{equation}\label{eqsolution}
  F ( v, E)= \A^{-1} M(v) = \int_0^\infty \e^{ -z} M ( v - E z) \d z.
\end{equation}

In the general case, it does not seem possible to get such an explicit formula.
However, Assumption \eqref{sigmaProp} and the normalization of $F$ imply
$$\nu_1 M(v)\leq  \K(F) \leq \nu_2 M(v).$$
In particular, $F$ satisfies
\begin{equation}\label{eq:FMFMF}
\nu_1 M(v)\leq  
\nu F+  E \cdot \nabla_v F  \leq \nu_2 M(v).
\end{equation}
As a consequence, we can prove the following proposition (see Theorem \ref{thm:FE} (ii)):
\begin{proposition}
There exist constants $C(R)$ and $c(R)$ such that if $|E|\leq R$ then
\begin{equation}\label{eq:FEbound}
 c(R) M(v)\leq F(v,E)\leq C(R) M(v) \qquad \mbox{ for all } v \in \R^d.
 \end{equation}
\end{proposition} 
This proposition follows immediately from \eqref{eq:FMFMF} and  the following lemma  (which will be used several times in this paper):

\begin{lemma}\label{definqbo}
There exist two constants $C(R) > 0$ and $c(R)>0$ 
such that if $|E|\leq R$ then the following holds:
\item[(i)] If $f$ satisfies 
\begin{equation}\label{eq:fMM}
  \nu f + E \cdot \nabla_v f  \leq \beta M
\end{equation}
for some $\beta > 0$, then
\[    
  f \leq C \beta M.   
\]
\item[(ii)] If $f$ satisfies 
\begin{equation}\label{eq:fMM2}
  \nu f + E \cdot \nabla_v f  \geq \beta M
\end{equation}
for some $\beta > 0$, then
$$ 
  f \geq c \beta M.  
$$
\end{lemma}
  
\begin{remark}\label{frem}
A similar result holds if we replace $M$ by $M( v) / ( 1 + | v|)$  in both inequalities. 
\end{remark} 

\begin{proof}[Proof of Lemma \ref{definqbo}]
Integrating \eqref{eq:fMM} (see the definition of $\A^{-1}$ given by \eqref{invfor}), we obtain 
\begin{align*}
f(v) & 
\leq \beta \int_0^\infty \e^{ - \int_0^z \nu ( v -   E \tau) \d \tau}  M ( v -   E z) \d z\\
&\leq \beta \int_0^\infty \e^{ - \nu_1 z} M ( v - E z) \d z, 
\end{align*}
and the first part of the lemma follows from the following claim:
There exists $C(R)>0$ such that
\begin{equation}
\label{eq:boundM}  \int_0^\infty \e^{ -\nu_1 z} M( v - E z) \d z \leq C(R) M( v) \qquad \text{for all }  v\in \R^d, \text{ and all } |E|\leq R .
\end{equation}
In order to prove \eqref{eq:boundM}, we first write
 \begin{align}
 \int_0^\infty \e^{ -\nu_1 z} \frac{ M( v - E z)}{ M( v)} \d z & = \int_0^\eta \e^{ -\nu_1 z} \frac{ M( v - E z)}{ M( v)} \d z + \int_\eta^\infty \e^{ -\nu_1 z} \frac{ M( v - E z)}{ M( v)} \d z \nonumber \\   
                                                        & = I_1 + I_2, \nonumber 
 \end{align}
 where $\eta = | v| / ( 2 |E|)$. The triangle inequality gives $| | v | - | E| z | \leq | v - E z|$, which implies
 \[
  \bigg| \frac{ v}{ 2} \bigg|^{ d + \alpha} \leq \Big| | v| - | E | z \Big|^{ d + \alpha} \leq | v - E z|^{ d + \alpha} \, ,
  \qquad \mbox{for $0 \leq z \leq \eta$.}
 \]
 Hence, using \eqref{comparison} yields 
$$
 \frac{ M( v - E z)}{ M( v)} \leq \frac{\mu_2}{\mu_1}
  \frac{ 1 + | v|^{ d + \alpha}}{ 1 + | v - E z|^{ d + \alpha}} \leq \frac{\mu_2}{\mu_1} \frac{ 1 + | v|^{ d + \alpha}}{ 1 + | v/ 2|^{ d + \alpha}}  \qquad \mbox{for $0 \leq z \leq \eta$.}
$$
Therefore we deduce
 \begin{align}
 I_1 = \int_0^\eta \e^{ -\nu_1 z} \frac{ M( v - E z)}{ M( v)} \d z & \leq \frac{ \mu_2}{ \mu_1} \int_0^\eta \e^{ -\nu_1 z} \frac{ 1 + | v|^{ d + \alpha}}{ 1 + | v/ 2|^{ d + \alpha}} \d z \nonumber \\
                                                    & \leq \frac{ \mu_2}{ \mu_1\nu_1} \frac{ 1 + | v|^{ d + \alpha}}{ 1 + | v/ 2|^{ d + \alpha}} \, \nonumber \\
                                                    & \leq C_1, \nonumber 
 \end{align}
 where $C_1 > 0$ does not depend on $v$. 
 Next, using \eqref{comparison} again, we get
 \begin{align}
I_2 = \int_{ \eta}^\infty \e^{ -\nu_1 z} \frac{ M( v - E z)}{ M( v)} \d z & \leq \frac{ \mu_2}{ \mu_1\nu_1} ( 1 + | v|^{ d + \alpha}) \e^{ -\nu_1 | v| / ( 2 | E |)} \nonumber \\
                                                            & \leq \frac{ \mu_2}{ \mu_1\nu_1} ( 1 + | v|^{ d + \alpha}) \e^{ -\nu_1  | v| /( 2R)} \nonumber \\
                                                            & \leq C_2, \nonumber 
 \end{align}
 where $C_2 > 0$ does not depend on $v$ (but depends on $R$). We thus obtain 
$$ \int_0^\infty \e^{ -\nu_1 z} \frac{ M( v - Ez)}{ M( v)} \d z \leq C_1+C_2
$$
which gives \eqref{eq:boundM} and completes the proof of the first part of the lemma.

\medskip

The second part of the lemma is somewhat easier to show. Indeed, proceeding as above, 
we check that \eqref{eq:fMM2} implies
\begin{align*}
f(v) & \geq \beta \int_0^\infty \e^{ - \int_0^s \nu ( v -  E \tau) \d \tau}  M ( v -  E s) \d s\\
     & \geq \beta \int_0^1 \e^{ - \nu_2 s} M ( v -  E s) \d s. 
\end{align*}
Furthermore, it is readily seen that there is a constant $c(R)$ such that 
$$ M(v-w )\geq c \, M(v) \, \quad\mbox{ for all } v, \; w\in\R^d,\quad |w|\leq R.$$
We deduce
\begin{align*}
f(v) 
&\geq c \beta \int_0^1 \e^{ - \nu_2 s} M ( v ) \d s\\
&\geq c \beta  M ( v ), 
\end{align*}
and the result follows.
\end{proof}

\subsection{Coercivity of the operator $\T$}
As a consequence of the results of the previous sections, we can now establish the following  coercivity 
property of  $\T$, which
will play a crucial role in this paper:
\begin{proposition}\label{prop:coer}
For all  $E\in \R^d$, the operator $\T$ defined by \eqref{eq:TE} satisfies
$$
  \int \T ( f)(v) \frac{ f(v)}{ F(v,E)} \d v \geq 0.
$$
Furthermore, 
for all $R>0$
there exists a constant $\theta(R)>0$ such that
for all $|E|\leq R$, there holds
  \begin{equation}
  \int \T ( f)(v) \frac{ f(v)}{ F(v,E)} \d v \geq \theta(R) \lVert f - \rho_f F \lVert^2_{ L^2_ {F^{-1}}(\R^d)}, \quad \text{ for all } f \in L^2_ {F^{-1}}(\R^d) \, . \label{coerineq}
 \end{equation}
\end{proposition}

\begin{proof}
Throughout this proof, we use the notation $f$ for $f(v)$ and $f'$ for $f(v')$ (and similar notations for $F$ and $M$).

Let us start by noting  the following
 \begin{align}
  \int \T ( f) \frac{ f}{ F} \d v & =\int E \cdot \nabla_v f \frac{ f}{ F } \d v +\int \nu \frac{ f^2}{ F} \d v -
   \int \int \sigma ( v, v') M f' \frac{ f}{ F} \d v \d v' \nonumber \\
                                 & =  \int \frac{ 1}{ 2} E \cdot \frac{ \nabla_v f^2}{ F} \d v 
                                 +\int \nu \frac{ f^2}{ F} \d v 
                                 -\int \int \sigma ( v, v') M F' \frac{ f'}{ F'} \frac{ f}{ F} \d v \d v'  \, . \nonumber
 \end{align}
 Integrating by parts and using the identity $  E \cdot \nabla_v F = \K ( F ) - \nu F $ we see that
 \begin{align}
 \frac{ 1}{ 2} \int E \cdot \frac{ \nabla_v f^2}{ F} \d v & = - \frac{ 1}{ 2} \int f^2   E \cdot \nabla_v \Big( \frac{ 1}{ F} \Big) \nonumber \\
                                                                                & = \frac{ 1}{ 2} \int \frac{ f^2}{ F^2} \left( \K ( F) - \nu F \right) \d v \, . \nonumber  
 \end{align}
 Using the fact that $M$ and $F$ are normalized functions and that $\sigma$ is symmetric, we deduce the following: 
 \begin{align}
  \int \T  ( f) \frac{ f}{ F } \d v & = \frac{ 1}{ 2} \int \nu \frac{ f^2}{ F } \d v + \frac{ 1}{ 2} \int \int \sigma ( v, v') M F ' \frac{ f^2}{ F ^2} \d v \d v'  \\
  & \qquad -  \int \int \sigma ( v, v' ) M F ' \frac{ f'}{ F '} \frac{ f}{ F} \d v \d v' \nonumber \\
                                 & =  \frac{ 1}{ 2} \int \int \sigma ( v', v) M' F  \frac{ f^2}{ F^2} \d v \d v' +  \frac{ 1}{ 2} \int \int \sigma ( v, v') M F ' \frac{ f^2}{ F ^2} \d v \d v'  \\
                                 & \qquad  - \int \int \sigma ( v, v') M F ' \frac{ f'}{ F'} \frac{ f}{ F } \d v \d v' \nonumber \\
                                 & = \frac{ 1}{ 2} \int \int \sigma ( v, v' ) \left(  M F' \bigg( \frac{ f'}{ F'} \bigg)^2 + M F' \frac{ f^2}{ F ^2} - 2 M F ' \frac{ f'}{ F '} \frac{ f}{ F } \right) \d v \d v' \nonumber \\
                                 & =  \frac{ 1}{ 2} \int \int \sigma ( v, v') M F' \bigg( \frac{ f}{ F} - \frac{ f'}{ F'} \bigg)^2 \d v' \d v . \nonumber
 \end{align}
 Since the right hand side is clearly non-negative, this gives the first inequality in the proposition.
 
If we further assume that $|E|\leq R$, then we can use
\eqref{eq:FEbound} and together with assumption \eqref{sigmaProp} it yields:
$$   \int \T ( f) \frac{ f}{ F} \d v
 \geq  \frac{ \nu_1}{ 2 C(R)} \int \int  F F' \bigg( \frac{ f}{ F} - \frac{ f'}{ F'} \bigg)^2 \d v' \d v .
$$
Finally, using the  decomposition $f = \rho_f F + g$ and the fact $\int_{\R^d} g \d v = 0$ we obtain 
 \begin{align}
  \int \T ( f) \frac{ f}{ F} \d v & \geq  \frac{ \nu_1}{ 2 C(R)} \int \int F F' \bigg( \frac{ g}{ F} - \frac{ g'}{ F'} \bigg)^2 \d v' \d v \nonumber \\
                                  & = \frac{ \nu_1}{ 2 C(R)} \int \int F \frac{ g'^2}{ F'} - 2 g g' + \frac{ g^2}{ F} F' \d v \d v' \,  \nonumber \\
                                  & = \frac{ \nu_1}{ C(R)} \int  \frac{ g^2}{ F}   \d v .\nonumber 
 \end{align}
This completes the proof. 
\end{proof}

\subsection{Properties of $F(v,E)$: Theorem \ref{thm:FE} (iii)}
This Section is devoted to the
 proof of the estimate on the derivative of $F$ with respect to $E$ (Theorem \ref{thm:FE}-(iii)).

First, we prove the following result.
  \begin{lemma}\label{lem:2}
For all $R>0$ there exists $C(R)$ such that
the function $F(v,E)$ solution of \eqref{eq:kernelF} satisfies 
  \begin{equation}\label{nabF}
   | \nabla_v F ( v,E) | \leq C(R) \frac{ M( v)}{ 1 + | v|},\qquad \mbox{ for all } v\in \R^d, \; |E|\leq R.
  \end{equation}
 \end{lemma}

 \begin{proof}
Differentiating \eqref{eq:kernelF}, with respect to $v_i$, we obtain:  
  \begin{align}
   E \cdot \nabla_v \left( \partial_{ v_i} F \right) + \nu \left( \partial_{ v_i} F \right) & = \int \sigma ( v, v') F ( v') \d v' \, \partial_{ v_i} M ( v) \nonumber \\
            & \quad  + \int \partial_{ v_i} \sigma ( v, v') F  ( v') \d v' M( v)- \left( \partial_{ v_i} \nu \right) F .  \label{eq:Fprop}                                                                                 
  \end{align}
The first term in the right hand side of \eqref{eq:Fprop} can be bounded by $C M( v) / ( 1 + |v|)$,  thanks to \eqref{sigmaProp} and 
  assumption \eqref{derdecay2}. The second term in \eqref{eq:Fprop} can also be bounded by $C M( v) / ( 1 + |v|)$
  thanks to the assumption \eqref{nuprop} and the normalization of $F$. 
  Finally,  using \eqref{nuprop} and \eqref{eq:FEbound}, the third term in the right hand side of \eqref{eq:Fprop} can 
  also be bounded by $C M( v) / ( 1 + |v|)$. We thus have 
$$\big|   E \cdot \nabla_v \left( \partial_{ v_i} F  \right) + \nu \left( \partial_{ v_i} F  \right) \big| \leq  C \frac{ M( v)}{ 1 + | v|}
$$
and we conclude the proof  using Lemma \ref{definqbo} and Remark \ref{frem}. 
 \end{proof}

We can now complete the proof of Theorem \ref{thm:FE}:
  
\begin{proof}[Proof of Theorem \ref{thm:FE}-(iii)] 
  We first prove that $\partial_E F$ is uniformly bounded in $L^2_{ F^{ -1}}$ for $|E|\leq R$: 
  Differentiating  \eqref{eq:kernelF}  
  with respect to $E_i$ yields:
      \begin{equation}\label{eq:auxF}
   \T( \partial _{E_i} F)  = -  \partial_{v_i} F.
   \end{equation}
   Thus multiplying by $\partial _{E_i} F / F$ and using the coercivity inequality \eqref{coerineq} (assuming $|E|\leq R$) we obtain
   \[
    \vartheta \| \partial _{E_i} F  \|^2_{ L^2_{ F^{ -1} }} \leq -  \int \partial_{v_i} F  \frac{ \partial _{E_i} F}{ F} \d v,
   \]
   where we have used the fact that $\partial _{E_i} \int F  \d v = 0$. The right hand side can be estimated using \eqref{nabF} and \eqref{eq:FEbound}:
   \begin{align}
 \left|   \int \partial_{v_i} F \frac{ \partial _{E_i} F  }{ F } \d v \right|& \leq C \int | \partial _{E_i} F | \d v  \leq C   \left( \int \frac{ | \partial _{E_i} F |^2}{ F} \d v \right)^{ 1 / 2}. \nonumber 
   \end{align}
   We deduce
   \[
    \vartheta \| \partial _{E_i} F   \|_{ L^2_{ F^{ -1}} } \leq C 
   \]
 which implies   in particular 
   \begin{equation}\label{eq:aus}
    \int | \partial _{E_i} F  | \d v \leq \left( \int \frac{ | \partial _{E_i} F |^2}{ F } \d v \right)^{ 1 / 2} \leq C  .
       \end{equation}
   
   Finally, in order to obtain \eqref{pEF} we
   rewrite  \eqref{eq:auxF} as    
   \begin{align}
    E \cdot \nabla_v  \partial _{E_i} F  + \nu  \partial _{E_i} F &= \K (\partial _{E_i} F ) -  \partial_{v_i} F\nonumber \\
                                                                          & =: H  ( v, E) \nonumber 
   \end{align}
and,   using the fact that $\int \partial _{E_i} F \, dv=0$, we note that
   $$H(v,E) =  \int [\sigma ( v, v') -\nu_0] \partial _{E_i} F ( v',E) \d v' M( v) -  \partial_{v_i} F$$
 So using \eqref{eq:silim}, \eqref{nabF} and \eqref{eq:aus}, we deduce
   \begin{align*}
    | H ( v,E)| & \leq  \int |\partial _{E_i} F ( v',E) |\d v' \frac{M( v)}{1+|v|} + C\frac{ M( v)}{1+|v|} \nonumber \\ 
                    & \leq C\frac{ M( v)}{1+|v|} .
   \end{align*}  
We can then conclude the proof using  Lemma \ref{definqbo} (see Remark \ref{frem}) and  \eqref{eq:FEbound}. 
\end{proof}

   
   

\subsection{Properties of $F(v,E)$: Proposition \ref{prop:Fas}}
When $\sigma=1$, we see, 
using \eqref{eqsolution} that 
\begin{equation}\label{eq:Fsimple}
F(v,E)\sim M(v) -   E\cdot \nabla _v M(v) \quad\mbox{ as } |E|\to 0.
\end{equation}
In the general case, we do not have an explicit formula for $F$ which would give us such an expansion. 
Our goal in this section is thus to prove Proposition \ref{prop:Fas} which gives the require asymptotic behavior of $F$ as $E$ goes to zero.

But first, we need to prove the existence of the auxiliary function $\lambda(v)$ appearing in 
\eqref{eq:D} and \eqref{eq:FGE}:
\begin{lemma}\label{lem:lambda0}
Assume \eqref{eqprop}-\eqref{nuprop}.
Then there exists a unique function $\lambda \in (L^2_{M^{-1}}(\R^d))^d$ satisfying
\begin{equation}\label{eq:lambda} 
Q(\lambda)(v) =\nabla_v M(v) , \qquad \int_{\R^d} \lambda(v) \, dv =0.
\end{equation}
Furthermore, it satisfies
\begin{equation}\label{eq:lambdaprop}
|\lambda(v)|\leq CM(v), \qquad |\partial_{v_i} \lambda_j(v)|\leq CM(v) \, \mbox{ for all $1 \leq i,j \leq d$}.
\end{equation}
\end{lemma}

We will first prove Proposition \ref{prop:Fas} and then go back to Lemma \ref{lem:lambda0}.

\begin{proof}[Proof of Proposition \ref{prop:Fas}]
We  define 
$$  G(v,E) :=  F(v,E) - M(v)  - E\cdot \lambda(v) .$$
It solves
\begin{align}
\T (G) &  = 0 - \T  (M) - E \cdot \T(\lambda)\nonumber \\
& = - E \cdot \nabla_v M -   E \cdot  (-Q(\lambda)+ E \cdot \nabla_v \lambda) \nonumber \\
& =  - E \cdot  (E\cdot \nabla_v \lambda), \label{eq:Geq}
\end{align}
and thus we obtain in particular
$$ \|\T (G)\|_{L^2_{F ^{-1}}} \leq  \  |E|^2 \| D_v \lambda \|_{L^2_{F^{-1}}}. 
$$
If $|E|\leq 1$, then inequalities \eqref{eq:FEbound} and \eqref{eq:lambdaprop} give
$$ \| D_v \lambda \|^2_{L^2_{F^{-1}}} \leq C \int  \frac{M(v)^2}{F(v,E)}\, \d v \leq \frac{C}{c} \int M(v)\,\d v \leq   C$$
and so
$$ \|\T(G)\|_{L^2_{F ^{-1}}} \leq  C |E|^2 . 
$$
Using the coercivity inequality \eqref{coerineq} (recall that $|E|\leq 1$), and the fact that $\int_{\R^d} G \d v= 0$, 
we deduce
\begin{align*}
\|G\|^2_{L^2_{F ^{-1}}}  = \int \frac{|G |^2}{F }\d v & \leq \frac{1}{\theta} \int \T(G ) \frac{G }{F }\, \d v \\
& \leq \frac{1}{\theta} \| \T(G)\|_{L^2_{F^{-1}}}\|G\|_{L^2_{F^{-1}}}
\end{align*}
and so
$$\|G\|_{L^2_{F^{-1}}} 
 \leq   \frac{1}{\theta}
  \|\T(G)\|_{L^2_{F^{-1}}}
  \leq
  \frac{C}{\theta}|E|^2  , 
$$
which gives \eqref{eq:GL22}.

Finally, using \eqref{eq:Geq} and the definition of $\T$, we write 
$$ \nu G + E\cdot \nabla_v G = K(G) -  E \cdot  (E\cdot \nabla_v \lambda). 
$$
Thanks to \eqref{eq:GL22} we obtain
$$ | K( G) | \leq  \|G\|_{L^2_{F^{-1}}}  M(v)\leq C   |E|^2  M(v),
$$
which implies, using \eqref{eq:lambdaprop}, the following estimate:
$$ |  \nu G +   E\cdot \nabla_v G | \leq  C   |E |^2 M(v). 
$$
We conclude the proof by applying Lemma \ref{definqbo}.
\end{proof}
 
Finally, we end this section with a proof of  Lemma \ref{lem:lambda0} which states the existence of the function $\lambda(v)$:
\begin{proof}[Proof of Lemma \ref{lem:lambda0}]
The existence and uniqueness of $\lambda$ follows from the coercivity of the operator $Q$ (see Lemma \ref{lem:Q}) and the fact that 

$$ \int_{\R^d} \nabla_v M(v) \, dv =0.
$$
Using Lemma \ref{lem:Q} together with \eqref{derdecay2} we obtain

\begin{equation}\label{eq:pk}
\| \lambda\|_{L^2_{M^{-1}}} \leq \frac 1{\nu_1} \| \nabla M \|_{L^2_{M^{-1}}} \leq \frac{ C}{ \nu_1}.
\end{equation}
Next, we rewrite \eqref{eq:lambda} as
\begin{align}
 \lambda(v) & = \frac{1}{\nu(v)}\left( \K(\lambda)(v) - \nabla_v M(v)\right) \nonumber\\
 & =  \frac{1}{\nu(v)}\left( \int \sigma(v,v') \lambda(v')\, dv ' M(v) - \nabla_v M(v)\right), \label{eq:lambda2} 
 \end{align}
and use \eqref{eq:pk} together with \eqref{derdecay2} to deduce the first inequality in \eqref{eq:lambdaprop}.

Finally, differentiating \eqref{eq:lambda2} with respect to $v$ and using 
\eqref{nuprop}
and 
\eqref{derdecay2}, we easily deduce
the second inequality in  \eqref{eq:lambdaprop}.
\end{proof}

\section{A priori estimates}
In this section we derive the a priori estimates on $f_\eps$ solution of \eqref{kineticeq} which will be necessary for the proofs of Theorems \ref{thm:main} and \ref{thm:2}.

First, we introduce the operator
\begin{equation}\label{opTe}
 \T_\eps ( f) : = - Q ( f) + \eps^{ \alpha - 1} E \cdot \nabla_v f,
 \end{equation}
and we recall that $F_\eps(x,v,t)$ denotes the solution of 
$$ \T_\eps(F_\eps) =0 \, \qquad \int_{\R^d} F_\eps(x,v,t)\, dv =1.$$
In view of Theorem \ref{thm:FE} (i), such a function exists and can be written as
$$ F_\eps (x,v,t) = F(v,\eps^{\alpha-1} E(x,t)).$$
When $\alpha\geq 1$ and $E$ satisfies \eqref{eq:E}, Theorem \ref{thm:FE} (ii) implies:
\begin{proposition}\label{prop:FMbd}
Assume that $\alpha\geq 1$.
Then there exists two positive constants $\gamma_1$ and $\gamma_2$ such that
for all $0<\eps\leq 1$, the following holds:
$$ \gamma_1 M(v)\leq F_\eps(x,v,t) \leq \gamma_2 M(v).$$
\end{proposition}
Under the same conditions, Theorem \ref{thm:FE} (iii) and the chain rule imply:
\begin{proposition}\label{FpropGe}
 Assume that $\alpha\geq 1$.
Then  for all $\eps\leq 1$, the function $F_\eps$ satisfies: 
 \begin{itemize}
  \item[\emph{(i)}] $\displaystyle \bigg\lVert \frac{ \partial_t F_\eps}{ F_\eps} \bigg\lVert_{ L^\infty \left( \R^{ 2d} \times [ 0, \infty) \right)} \leq C \eps^{ \alpha - 1}  ,$
  \item[\emph{(ii)}] $\displaystyle \bigg\lVert \frac{ v \cdot \nabla_x F_\eps}{ F_\eps} \bigg\lVert_{ L^\infty \left( \R^{ 2d} \times [ 0, \infty) \right)} \leq C \eps^{ \alpha - 1}  ,$ 
 \end{itemize} 
 where $C$ is a positive constant depending on $\|E\|_{W^{1,\infty}}$ but not on $\eps$.
\end{proposition}
\begin{proof}
We only prove the second inequality (the first one is easier): We have 
$$ v \cdot \nabla_x F_\eps = \partial_E F(v,\eps^{\alpha-1} E(x,t)) \eps^{\alpha-1} v \cdot \nabla_x E
$$
and so \eqref{pEF} and the fact that $\alpha\geq 1$ implies
$$ 
|v \cdot \nabla_x F_\eps| \leq C F_\eps \frac{ \eps^{\alpha-1} v \cdot \nabla_x E}{1+|v|} \leq  C  \eps^{\alpha-1} \| \nabla E\|_{L^\infty} F_\eps ,
$$
which proves (ii).
\end{proof}
Finally, Proposition \ref{prop:coer} implies
\begin{proposition}\label{prop:coere}
Assume that $\alpha\geq 1$.
Then for all $\eps\leq 1$ there holds
  \begin{equation}
  \int \T_\eps ( f)(v) \frac{ f(v)}{ F_\eps} \d v \geq \theta(R) \lVert f - \rho_f F_\eps \lVert^2_{ L^2_ {F_\eps^{-1}}(\R^d)}, \quad \text{ for all } f \in L^2_ {F_\eps^{-1}}(\R^d) \, . \label{coerineqe}
 \end{equation}
\end{proposition}

We can now prove the main result of this section:
\begin{proposition}\label{apriorieq1}
Assume that $\alpha \in [1,2)$ and that 
\eqref{eq:E}-\eqref{nuprop} hold. Let $f_\eps$ be the 
solution of  \eqref{kineticeq}
and let $\rho_\eps(x,t)= \int_{\R^d} f_\eps(x,v,t)\d v$.
Then: 
\begin{itemize}
 \item[ \emph{(i)} ] The sequence $\left( f_\eps \right)$ is  bounded uniformly with respect to $\eps$ in $L^\infty \left( \left( 0, \infty \right); L^1 \left( \R^d \times \R^d \right) \right)$ and 
$\left( \rho_\eps \right)$ is bounded uniformly with respect to $\eps$  in $L^\infty \left( \left( 0, \infty \right); L^1 \left( \R^d \right) \right)$. 
 \item[ \emph{(ii)} ] For all $T>0$, 
 $( f_\eps )$ is  bounded uniformly with respect to $\eps$ in $L^\infty \left( ( 0, T); L^2_{ M^{ -1}} \left( \R^{ 2d} \right) \right)$,
 and 
$( \rho_\eps )$ is bounded uniformly with respect to $\eps$  in  $L^\infty \left( ( 0, T); L^2 \left( \R^d \right) \right)$. 
 \item[ \emph{(iii)} ] The function $f_\eps$ can be decomposed as 
 $f_\eps = \rho_\eps F_\eps + g_\eps$ where 
 $g_\eps$ satisfies 
\begin{equation}\label{eq:boundg}
\| g_\eps\|_{ L^2 ( ( 0, T), L^2_{ M^{ -1}} ( \R^{ 2d}))} \leq C(T) \eps^{\alpha/2}.
\end{equation}
\end{itemize}
\end{proposition}

\begin{proof}
Integrating \eqref{kineticeq} with respect to $x$ and $v$ and thanks to the conservation of mass property of the operator $Q$ we obtain
that $(f_\eps)$ is uniformly bounded in $L^\infty ( \left( 0, \infty \right); L^1( \R^{ 2d}) )$. Next using \eqref{opTe}, we recast \eqref{kineticeq} as
$$
\varepsilon^\alpha \partial_t f_\varepsilon + \varepsilon v \cdot \nabla_x f_\varepsilon +  \T_\eps ( f_\varepsilon) = 0.
$$
Multiplying this equation
by $f_\eps / F_\eps$ and integrating with respect to $x$ and $v$ we get:
\begin{align}
 \frac{ \eps^\alpha}{ 2} \frac{ \d}{ \d t} \lVert f_\eps \lVert^2_{ L^2_{ F^{ -1}_\eps} ( \R^{ 2d}) } & = - \frac{ \eps^{ \alpha}}{ 2} \int \int \frac{ \partial_t F_\eps}{ F_\eps} \frac{ f^2_\eps}{ F_\eps} \d v \d x
 + \frac{ \eps}{ 2} \int \int \frac{ v \cdot \nabla_x F_\eps}{ F_\eps} \frac{ f^2_\eps}{ F_\eps} \d v \d x \nonumber \\
                                                                         & \quad + \int \int \T_\eps ( f_\eps) \frac{ f_\eps}{ F_\eps} \d v \d x. \nonumber
\end{align} 
Using  \eqref{coerineqe} and Proposition \ref{FpropGe}, we deduce
\begin{equation}\label{eq:apriori}
 \frac{ \eps^\alpha}{ 2} \frac{ \d}{ \d t} \lVert f_\eps \lVert^2_{ L^2_{ F^{ -1}_\eps} ( \R^{ 2d}) }+ \theta \lVert f_\eps - \rho_{\eps} F_\eps \lVert^2_{ L^2_{ F^{ -1}_\eps} ( \R^{ 2d}) } \leq \eps^{ \alpha} C \lVert f_\eps \lVert^2_{ L^2_{ F^{ -1}_\eps} ( \R^{ 2d}) } .
\end{equation}

In particular this yields
\[
 \frac{ \d}{ \d t} \lVert f_\eps \lVert^2_{ L^2 ( \d v \d x /  F_\eps)} \leq 2 C \lVert f_\eps \lVert^2_{ L^2 ( \d v \d x /  F_\eps)},
\]
and Gronwall's Lemma implies that $(f_\eps)$ is uniformly bounded in 
$L^\infty \left( ( 0, T ); L^2_{ F_\eps^{ -1}} ( \R^{ 2d} ) \right)$ for any $T>0$
and thus in $L^\infty \left( ( 0, T ); L^2_{ M^{ -1}} ( \R^{ 2d}) \right)$ thanks to Proposition \ref{prop:FMbd}.
We also deduce that
\[
 \int \rho^2_\eps \d x = \int \left( \int f_\eps \d v \right)^2 \d x \leq \int \int \frac{ f^2_\eps}{ F_\eps} \d v \d x \leq C. 
\]



\noindent Finally, integrating \eqref{eq:apriori} with respect to $t$ and using Proposition \ref{prop:FMbd}, we obtain \eqref{eq:boundg}.

\end{proof}

\section{Proof of Theorem \ref{thm:main}}
The proof of our main result relies on  the test function method first introduced in \cite{MR2815035}.
The starting point of the method
is the introduction of the following auxiliary test function:
Given $\phi(x,t) \in \mathcal D ( \R^N \times [0,\infty))$, we denote by $\chi_\varepsilon(x,v,t)$ the unique bounded solution of the auxiliary problem
\begin{equation}\label{auxiliaryeq}
\nu(v)\chi_\eps - \varepsilon v \cdot \nabla_x \chi_\eps = \nu(v)\phi \,,
\end{equation}
which (integrating \eqref{auxiliaryeq} along the characteristics) yields:
\begin{equation}\label{explicitaux}
\chi_\varepsilon ( x, v, t) = \int_0^\infty \e^{-\nu(v)z}\nu(v) \phi ( x + \varepsilon v z, t) \d z \,. 
\end{equation}
We then have:
\begin{lemma}
Let $f_\eps$ be a weak solution of  \eqref{kineticeq} and let $\chi_\eps$ be given by \eqref{explicitaux}.
Then the following weak formulation holds:
 \begin{align}
 &  \int\int\int f_\eps \dt \chi_\eps \d v\d x\d t + \int\int f^{ in} \chi_\eps |_{ t=0} \d v\d x 
 + \eps^{-\alpha} \int \int\int \rho_\eps\nu  F_\eps (\chi_\eps - \phi)\d v\d x\d t  \nonumber \\ 
 & \qquad = -\eps^{-1} \int  \int\int g_\eps ( E \cdot \dv \chi_\eps) \d v\d x\d t 
-\eps^{-\alpha} \int  \int\int K(g_\eps) (\chi_\eps -\phi)  \d v\d x\d t  
, \label{weakform}
 \end{align}
with
\begin{equation}\label{eq:g} 
g_\eps = f_\eps-\rho_\eps F_\eps, \qquad \rho_\eps = \int_{\R^d} f_\eps \, \d v.
\end{equation}
 \end{lemma}
\begin{proof}
Taking $\chi_\eps$ as a test function in  \eqref{kineticeq} and using \eqref{auxiliaryeq}, we get
 \begin{align*}
 & - \int \int\int f_\eps \dt \chi_\eps \d v\d x\d t - \int\int f^{ in} \chi_\eps |_{ t=0} \d v\d x  \\
 & = \eps^{-1}  \int \int\int f_\eps E\cdot \nabla_v \chi_\eps \d v\d x\d t
+\eps^{-\alpha}\int \int\int K(f_\eps) \chi_\eps - \nu f_\eps \phi \d v\d x\d t\\
 & = \eps^{-1}  \int \int\int f_\eps E\cdot \nabla_v \chi_\eps \d v\d x\d t
+\eps^{-\alpha}\int \int\int K(f_\eps) (\chi_\eps - \phi) \d v\d x\d t,
\end{align*} 
where we used the fact that $\int K(f)\, \d v = \int \nu f \d v$ for all $f$.
Using \eqref{eq:g}, we deduce:
 \begin{align*}
 & - \int \int\int f_\eps \dt \chi_\eps \d v\d x\d t - \int\int f^{ in} \chi_\eps |_{ t=0} \d v\d x  \\
 & = \eps^{-1}  \int \int\int \rho_\eps F_\eps E\cdot \nabla_v \chi_\eps \d v\d x\d t
+\eps^{-\alpha}\int \int\int \rho_\eps K(F_\eps) (\chi_\eps-  \phi) \d v\d x\d t\\
& \quad +  \eps^{-1}  \int \int\int g_\eps E\cdot \nabla_v \chi_\eps \d v\d x\d t
+\eps^{-\alpha}\int \int\int K(g_\eps) (\chi_\eps-  \phi) \d v\d x\d t.
\end{align*} 
Finally, using the definition of $F_\eps$ and  the fact that $\int K(F)\, \d v = \int \nu F \d v$, we find
\begin{align*}
&  \eps^{-1} \int   F_\eps E\cdot \nabla_v \chi_\eps \d v 
+\eps^{-\alpha} \int  K(F_\eps) (\chi_\eps-  \phi) \d v \\
& \qquad =
- \eps^{-1} \int    (E\cdot \nabla_v F_\eps) \chi_\eps \d v 
+\eps^{-\alpha} \int  K(F_\eps) (\chi_\eps-  \phi) \d v \\
&\qquad  =
- \eps^{-\alpha} \int  (K(F_\eps)-\nu F_\eps) \chi_\eps \d v 
+\eps^{-\alpha} \int  K(F_\eps) \chi_\eps- \nu F_\eps \phi \d v \\
&\qquad  =
 \eps^{-\alpha} \int \nu F_\eps (\chi_\eps-\phi) \d v
\end{align*}
which concludes the proof.
\end{proof}

In order to prove Theorem \ref{thm:main}
we need to show that the right hand side of \eqref{weakform} goes to zero, and to identify the limit of the left hand side. The first point 
follows from the following result.

\begin{proposition}\label{prop:2}
For any test function $\phi \in \mathcal D(\R^N\times [0,\infty))$, let $\chi_\eps$ be defined by  \eqref{explicitaux}. 
Then
$$\lim_{\eps\rightarrow 0}  \eps^{-\alpha} \int \int \int K(g_\eps) (\chi_\eps -\phi)  \, dx\, dv\,dt= 0,
$$
and
$$\lim_{\eps\rightarrow 0} \eps^{-1} \int  \int\int g_\eps  E \cdot \dv \chi_\eps \d v\d x\d t  = 0.
$$
\end{proposition}
We will give a proof of this proposition which holds for any $\alpha \in (0,2) $ (and not just $\alpha>1$), since we will use the result for $\alpha=1$ in the next section.
\begin{proof}
To prove the first convergence, we note that 
\begin{align*} |K(g_\eps)(x,t)| & = \left| \int\sigma(v,v') g_\eps(x,v',t)\, dv' \right| M(v) \\
& \leq \nu_2 \left( \int\frac{g_\eps(x,v',t)^2}{M(v')}\, dv'\right)^{1/2} M(v).
\end{align*}
Therefore
\begin{align}
\left| \eps^{-\alpha} \int \int \int K(g_\eps) (\chi_\eps -\phi)  \, dx\, dv\,dt \right|
& \leq \eps^{-\alpha} \nu_2 \int \int \| g_\eps(x,\cdot,t)\|_{L^2_{M^{-1}}} \int_{ \R^d} M(v)  | \chi_\eps -\phi|  \, dv\, dx\,dt\nonumber \\
& \leq \eps^{-\alpha} \nu_2 \| g_\eps \|_{L^2_{M^{-1}}} \left( \int \int \left( \int_{ \R^d} M(v)  |\chi_\eps -\phi|  \, dv\right)^2\, dx\,dt\right)^{1/2}, \label{eq:Kes}
\end{align}
and we conclude thanks to the following result:
\begin{lemma}\label{lem:eta}
For all $\phi\in\mathcal D(\R^d)$ and all  $\eta <\alpha$, there exists a constant $C$ depending on $\eta$  such that
$$
 \left( \int \left( \int_{ \R^d} M(v)  |\chi_\eps -\phi|  \, dv\right)^2\, dx \right)^{1/2} \leq C  \| \phi ( \cdot, t) \|_{H^1(\R^d)}  \eps^\eta.
$$
\end{lemma}
Postponing the proof of this lemma to the end of this proof, we deduce (using \eqref{eq:boundg} and the fact that $\phi(x,t)=0$ is compactly supported in $t$):
\begin{align*}
\left| \eps^{-\alpha} \int \int \int K(g_\eps) (\chi_\eps -\phi)  \, dx\, dv\,dt \right|
& \leq C \eps^{\eta-\alpha} 
\| g_\eps \|_{ L^2_{M^{-1} }( \R^{ 2d} \times (0,T))}  
\|  \phi\|_{ L^2(0,\infty;H^1( \R^d))}\\
& \leq C\|  \phi\|_{  L^2(0,\infty;H^1( \R^d))} \eps^{\eta-\alpha/2},
\end{align*}
and the result follows by choosing any $\eta\in(\alpha/2,\alpha)$.

\medskip

To prove the second limit, we first rewrite 
\eqref{explicitaux} as
$$
\chi_\varepsilon ( x, v, t) = \int_0^\infty \e^{-s}  \phi \left( x + \varepsilon \frac{v}{\nu(v)} z, t\right) \d s 
$$
and observe that
\begin{align}
\frac{ 1}{ \eps}  \partial_{v_i} \chi_\eps   & =  \int_0^\infty s \e^{ -s} \partial_{x_j}\phi \left( x + \varepsilon \frac{v}{\nu(v)} z, t\right) 
\partial_{v_i}\left(\frac{v_j}{\nu(v)}\right)
\d z. \label{eq:chidv} 
\end{align}
Next let us note that thanks to \eqref{nuprop2} we obtain
 $$ \left| \partial_{v_i}\left(\frac{v_j}{\nu(v)}\right)\right| \leq \frac{1}{\nu(v)} + \frac{|v|\, |\nabla_v \nu(v)|}{\nu(v)^2} \leq C,
 $$
 for all $1 \leq i, j \leq d$. Using Jensen's inequality, we deduce:
\begin{align*}
\int\int \left|\frac{ 1}{ \eps}  \dv \chi_\eps\right| ^2M(v)\, \d v \d x 
& \leq C \int\int   \int_0^\infty s \e^{ -s} \left| \dx \phi\left( x + \varepsilon \frac{v}{\nu(v)} z, t\right)  \right|^2 \d z \,
M(v)\, \d v \d x \\
& \leq C \| \nabla_x \phi ( \cdot, t)\| _{L^2(\R^d)}. 
\end{align*}
Therefore, by Cauchy-Schwarz
  \begin{align*}
\left| \eps^{-1} \int  \int\int g_\eps ( E \cdot \dv \chi_\eps) \d v\d x\d t \right|
& \leq  \| g_\eps\|_{L^2_{M^{-1}}} \| E\|_{L^\infty} \| \nabla_x \phi \| _{L^2(\R^d \times ( 0, \infty))}, 
\end{align*}
 which completes the proof thanks to \eqref{eq:boundg}.
\end{proof}

\begin{proof}[Proof of Lemma \ref{lem:eta}]
For any $\delta>0$ we can write:
\begin{align*}
 \left( \int_{ \R^d} M(v)  |\chi_\eps -\phi|  \, \d v\right)^2
 & \leq  C \left( \int_{ \R^d} \frac{1}{(1+|v|)^{ d+\alpha}}  |\chi_\eps -\phi|  \, \d v\right)^2\\
 & \leq  C \left( \int_{ \R^d} \frac{1}{(1+|v|)^{ d+\delta}}\, \d v \right)\left( \int_{ \R^d} \frac{1}{(1+|v|)^{ d+2\alpha-\delta}}  |\chi_\eps -\phi|^2  \, \d v\right)\\ 
 & \leq  C_\delta \int_{ \R^d} \frac{1}{(1+|v|)^{ d+2\alpha-\delta}}  |\chi_\eps -\phi|^2  \, \d v.
\end{align*}
Furthermore, we have
\begin{align*}
|\chi_\eps -\phi| & = 
\left| \int_0^\infty \e^{-\nu(v)z}\nu(v) [\phi ( x + \varepsilon v z) - \phi( x) ] \d z \right|\\
& \leq 
\left( \int_0^\infty \e^{-\nu(v)z}\nu(v) [\phi ( x + \varepsilon v z) - \phi( x) ]^2 \d z\right)^{1/2}
\end{align*}
and so
\begin{align*}
\int_{\R^d}
 |\chi_\eps -\phi|^2
\, \d x 
& \leq   \int_0^\infty \e^{-\nu(v)z}\nu(v) \int_{\R^d} [\phi ( x + \varepsilon v z) - \phi( x ) ]^2\, \d x  \d z. 
\end{align*}
Finally, using the inequalities
$$ \int_{\R^d} [\phi ( x + \varepsilon v z) - \phi( x) ]^2\, \d x \leq  2 \| \phi ( \cdot, t)\|^2_{L^2(\R^d)} 
$$
and 
$$ \int_{\R^d} [\phi ( x + \varepsilon v z) - \phi( x) ]^2\, \d x \leq  \| \nabla_x \phi ( \cdot, t)\|^2_{ L^2(\R^d)} |\eps v  z|^2$$
we note that for any $\eta\in(0,1)$ there exists a constant $C$ such that
$$ \int_{\R^d} [\phi ( x + \varepsilon v z) - \phi( x) ]^2\, \d x \leq C \| \phi ( \cdot, t)\|^2_{ H^1(\R^d)} (\eps |v| z)^{2\eta}.$$
We deduce
$$
\int \left( \int_{ \R^d} M(v)  |\chi_\eps -\phi| \, \d v\right)^2 \d x 
 \leq C  \eps ^{2\eta}  \| \phi ( \cdot, t) \|_{H^1(\R^d )} ^2 
 \int_{ \R^d}  \int_0^\infty \e^{-\nu(v)z}\nu(v) \frac{( |v| z)^{2\eta} }{(1+|v|)^{ d+2\alpha-\delta}}    \, \d z \d v
$$
where the last integral is finite provided we choose $\eta<\alpha$ and then $\delta < 2(\alpha-\eta)$.
\end{proof}

Having proved that the two terms in the right hand side of \eqref{weakform} go to zero as $\eps\to 0$, we now prove the following result, 
which  shows how the asymptotic equation appears when passing to the limit in \eqref{weakform}:
\begin{proposition}\label{prop:4}
Let $\mathcal L^\eps$ be the operator defined by
$$ \mathcal L^\eps(\phi)(x,t) :=   \eps^{-\alpha} \int \nu F_\eps (\chi_\eps-\phi) \d v $$
for all $ \phi \in \mathcal D(\R^d \times (0,\infty))$,
where $\chi_\eps$ is defined by \eqref{explicitaux}. Then 
$$ \mathcal L^\eps(\phi) \longrightarrow \mathcal L(\phi):= - \kappa ( - \Delta)^{ \alpha / 2} ( \phi) 
- (D E) \cdot \nabla_x \phi \quad \mbox{ as } \eps\to0$$
uniformly and in $L^2$. The matrix $D$ is defined by \eqref{eq:D} and $\kappa$ is given by \eqref{kapval}.
\end{proposition}

The key to the proof of this proposition is the following immediate consequence
of Proposition \ref{prop:Fas}:
\begin{proposition}\label{lem:Feps}
When $\alpha>1$, the function $F_\eps$ satisfies
\begin{equation} \label{eq:FG}
F_\eps (x,v,t) = M(v) + \eps^{\alpha-1} E(x,t)\cdot \lambda(v) + G_\eps(x,v,t)
\end{equation}
where $\lambda(v)$ is given by \eqref{eq:lambda}
and  $G_\eps$ satisfies:
\begin{equation}\label{eq:GLinfty} 
| G_\eps(x,v,t)| \leq  C \eps^{2(\alpha-1)} | E(x,t)|^2 M(v) \quad \mbox{ for all $(x,v,t)$}.
\end{equation}
\end{proposition}

\begin{proof}[Proof of Proposition \ref{prop:4}]
Using Proposition \ref{lem:Feps} above, we write
\begin{equation}\label{eq:L}
 \mathcal L^\eps(\phi) = L_1^\eps + L_2^\eps + L_3^\eps
 \end{equation}
where
\begin{align*}
 L^\eps_1 & =   \eps^{-\alpha} \int \nu M(v) (\chi_\eps-\phi) \d v\\
 L^\eps_2 & =   \eps^{-1} \int \nu  E(x,t)\cdot \lambda(v)(\chi_\eps-\phi) \d v\\
 L^\eps_3 & = \eps^{-\alpha} \int \nu G_\eps (\chi_\eps-\phi) \d v.
\end{align*}
The first term converges to  $- \kappa ( - \Delta)^{ \alpha / 2} ( \phi) $ uniformly and in $L^2$, as was proved, for instance in \cite{MR2815035}.

For the second term, we note that
$$
L^\eps_2  =   E(x,t)\cdot \left( \eps^{-1} \int \nu   \lambda(v)(\chi_\eps-\phi) \d v\right)
$$
and we conclude thanks to the following lemma (which is proved below):
\begin{lemma}\label{lem:lambda}
For any test function $\phi$, we have
$$
\lim_{\eps\to 0} \eps^{-1} \int_{\R^d} \nu   \lambda (v)(\chi_\eps-\phi) \d \, v=\int_{\R^d}  \lambda(v)  (v \cdot \nabla_{x} \phi ( x ,t)) \d \, v = D^T \nabla_x \phi
$$
where the limit holds uniformly and in $L^2$.
\end{lemma}

Finally, for the last term in \eqref{eq:L}, 
we write
\begin{align}
\chi_\eps -\phi 
 &= 
 \int_0^\infty \e^{-\nu(v)z}\nu(v) [\phi ( x + \varepsilon v z, t) - \phi( x, t) ] \d z \nonumber \\
 &  = \int_0^\infty \e^{-\nu(v)z}\nu(v) \int_0^z \eps v \cdot \nabla_x \phi ( x + \varepsilon v s, t)\d s \d z, \label{eq:chivphi}
\end{align}
which gives:
\begin{align*}
L^\eps_3 & = \eps^{-\alpha} \int_{\R^d} \nu G_\eps (\chi_\eps-\phi) \d v\\
& = \eps^{1-\alpha}  \int_{\R^d} \int_0^\infty  \int_0^z \e^{-\nu(v)z}\nu(v)^2   G_\eps   v \cdot \nabla_x \phi ( x + \varepsilon v s, t)\d s \d z \d v.
\end{align*}
Using  \eqref{eq:GLinfty}, we deduce
\begin{align*}
|L^\eps_3| & 
\leq  C    \eps^{\alpha-1} | E(x,t)|^2   \int_{\R^d} \int_0^\infty  \int_0^z \e^{-\nu(v)z}\nu(v)^2  M(v)    v \cdot \nabla_x \phi ( x + \varepsilon v s, t)\d s \d z \d v.
\end{align*}
Next, thanks to the fact that $\int_{\R^d} |v| M(v)\d v$ is finite (since $\alpha>1$) we obtain
\begin{align*}
\|L^\eps_3\|_{L^\infty(\R^d\times(0,\infty))} & 
\leq  C    \eps^{\alpha-1} | E(x,t)|^2\| \nabla_x \phi \|_{L^\infty},
\end{align*}
and applying Jensen's inequality we get
\begin{align*}
|L^\eps_3|^2 & 
\leq  C  (  \eps^{\alpha-1} | E(x,t)|^2  )^2 \int_{\R^d} \int_0^\infty  \int_0^z \e^{-\nu(v)z}\nu(v)^2  M(v)    |v| |\nabla_x \phi ( x + \varepsilon v s, t)|^2\d s \d z \d v,
\end{align*}
hence\begin{align*}
\|L^\eps_3\|_{L^2(\R^d\times(0,T))} & 
\leq  C    \eps^{\alpha-1} \| E(x,t)\|_{L^\infty}^2 \| \nabla_x \phi \|_{L^2(\R^d\times(0,T))}, 
\end{align*}
which completes the proof.
\end{proof}

\begin{proof}[Proof of Lemma \ref{lem:lambda}]
First, using \eqref{eq:chivphi} we obtain
\begin{equation}\label{eq:lambda11}
\eps^{-1} \int_{\R^d} \nu   \lambda_i(v)(\chi_\eps-\phi) \d v
=
\int_{\R^d} 
\int_0^\infty  \int_0^z \e^{-\nu(v)z}\nu(v)^2   \lambda_i(v)  v \cdot \nabla_x \phi ( x + \varepsilon v s, t)\d s \d z
\d v.
\end{equation}
Next, we note that for any $\delta\in(0,1)$, we have
$$ 
| \nabla_x \phi ( x + \varepsilon v s, t)-\nabla_x\phi(x,t)|\leq C |\eps v s|^\delta,
$$
and
$$
\int_0^\infty  \int_0^z \e^{-\nu(v)z}\nu(v)^2   \d s \d z=1
$$
where the first inequality follows from the two inequalities $|\nabla_x \phi(x+y)-\nabla_x\phi(x)|\leq C$ (for $|y|\geq 1$)
 and $|\nabla_x \phi(x+y)-\nabla_x\phi(x)|\leq C|y|$ (for $|y|\leq 1$). Hence, thanks
to \eqref{eq:lambdaprop} we deduce
\begin{align*}
& \left | \eps^{-1} \int_{\R^d} \nu   \lambda_i(v)(\chi_\eps-\phi) \d v - \int_{\R^d}  \lambda_i (v)  v \d v\cdot \nabla_{x} \phi ( x ,t)\right| \\
& \hspace{3cm} \leq 
C \int_{\R^d} 
\int_0^\infty  \int_0^z \e^{-\nu(v)z}\nu(v)^2   \lambda(v) | v| |\eps vs|^\delta \d s \d z
\d v\\
& \hspace{3cm} \leq 
C \eps^\delta \int_{\R^d} 
  M(v) | v|^{1+\delta} 
\d v.
\end{align*}
The uniform convergence follows by choosing $\delta $ such that $0<\delta<\alpha-1$.

\medskip

Finally, going back to \eqref{eq:lambda11}, we also deduce
\begin{align*}
& \int_0^\infty \int_{\R^d} \left| \eps^{-1} \int_{\R^d} \nu   \lambda(v)(\chi_\eps-\phi) \d v \right | \, dx\, dt\\
&\qquad  \leq 
\int_0^\infty \int_{\R^d} 
\int_{\R^d} 
\int_0^\infty  \int_0^z \e^{-\nu(v)z}\nu(v)^2   \lambda(v)  |v| | \nabla_x \phi ( x + \varepsilon v s, t)| \d s \d z
\d v \, dx\, dt\\
& \qquad 
 \leq \|  \nabla_x \phi\|_{ L^1 \left( \R^d\times( 0, T) \right)}
\int_{\R^d} 
\int_0^\infty  \int_0^z \e^{-\nu(v)z}\nu(v)^2   \lambda(v)  |v|  \d s \d z
\d v \\
&\qquad   \leq C \|  \nabla_x \phi\|_{ L^1 \left( \R^d\times( 0, T) \right)}.
\end{align*}
So by a simple interpolation, we see that since the quantity under consideration is bounded in $L^1$ and converges uniformly, it also converges in $L^2$.
\end{proof}

Gathering the results above, we can now complete the proof of Theorem \ref{thm:main}:

\begin{proof}[Proof of Theorem \ref{thm:main}]
In view of Proposition \ref{apriorieq1} and using a diagonal extraction argument, we can assume (up to a subsequence) that 
there exist two functions $f(x,v,t)$ and $\rho(x,t)$ such that
$$ f_\eps \rightharpoonup f \quad  \mbox{ in $L^\infty ( (0,T);L^2_{M^{-1}}(\R^{2d}) )$-weak $\star$}$$
 and 
$$\rho_\eps \rightharpoonup \rho\quad \mbox{ in $L^\infty((0,T);L^2(\R^{d}))$-weak $\star$} $$
for all $T>0$.
Furthermore, Proposition \ref{apriorieq1} (iii), together with Proposition \ref{lem:Feps} implies
$$ \| f_\eps - \rho_\eps M\|_{L^2(0,T;L^2_{ M^{ -1}} (\R^{2d}))} \leq C(T)\eps^{\alpha-1}$$
and so
$$ f(x,v,t)=\rho(x,t) M(v).$$ 

Next, we recall that Lemma \ref{lem:eta} gives:
$$  \int M [ \chi_\eps-\phi ] \d v \longrightarrow 0 \mbox{ in $L^2(\R^d\times(0,\infty))$},$$ 
and we can prove similarly that
$$  \int M [\partial _t \chi_\eps-\partial_t \phi ] \d v \longrightarrow 0 \mbox{ in $L^2(\R^d\times(0,\infty))$}.$$ 
Using these facts, it is easy to show that
\begin{align*}
& \lim_{\eps\to 0} \left( \int_0^\infty \int\int f_\eps \dt \chi_\eps \d v\d x\d t + \int\int f^{ in} \chi_\eps |_{ t=0} \d v\d x \right) \\
& \qquad\qquad \qquad\qquad =
 \int_0^\infty \int \rho \dt \phi  \d x\d t + \int\int \rho^{ in} \phi |_{ t=0} \d x .
\end{align*}

Finally combining this limit with Propositions \ref{prop:2} and \ref{prop:4}, we can now pass to the limit in \eqref{weakform} to deduce:
$$
 \int_0^\infty \int \rho \dt \phi  \d x\d t + \int\int \rho^{ in} \phi |_{ t=0} \d x 
 + 
 \int_0^\infty \int \rho \left[ - \kappa ( - \Delta)^{ \alpha / 2} ( \phi) 
- (DE) \cdot \nabla_x \phi \right]\d x \d t=0
$$
which is the weak formulation of \eqref{drifdifeq}.
\end{proof}

\section{Proof of Theorem \ref{thm:2}}
Before proving Theorem \ref{thm:2}, we need to show that $\mu(E)$ defined by \eqref{eq:mu} is well defined:
\begin{lemma}\label{lem:mu}
The function $R(v,E) = F(v,E) - M(v)$ satisfies
 \[
  | R ( v, E)| \leq C|E| \frac{ M( v)}{ 1 + | v|},
 \] 
For some constant $C > 0$.
In particular, the quantity $\mu(E)$ defined by \eqref{eq:mu} is well defined for all $E\in\R^d$
and satisfies $|\mu(E)|\leq C|E|$.
\end{lemma}

Postponing the proof of this lemma to the end of this section, we turn to the proof of Theorem~\ref{thm:2}:

\begin{proof}[Proof of Theorem \ref{thm:2}]
When $\alpha = 1$, $F_\eps(x,v,t)=F(v,E(x,t)$ is independent of $\eps$ (we thus drop the $\eps$ subscript below)
and
 the weak formulation \eqref{weakform} takes the form
\begin{align}
 &  \int\int\int f_\eps \dt \chi_\eps \d v\d x\d t + \int\int f^{ in} \chi_\eps |_{ t=0} \d v\d x 
 + \frac{ 1}{ \eps} \int \int\int \rho_\eps\nu  F (\chi_\eps - \phi)\d v\d x\d t  \nonumber \\ 
 & \qquad = - \frac{ 1}{ \eps} \int  \int\int g_\eps ( E \cdot \dv \chi_\eps) \d v\d x\d t 
- \frac{ 1}{ \eps} \int  \int\int K(g_\eps) (\chi_\eps -\phi)  \d v\d x\d t. \label{weakform3}
\end{align}
Proceeding as in the proof of Theorem \ref{thm:main}, we have (see Proposition \ref{apriorieq1}): 
$$ f_\eps \rightharpoonup f \quad  \mbox{ in $L^\infty ( (0,T);L^2_{M^{-1}}(\R^{2d}) )$-weak $\star$}$$
 and 
$$\rho_\eps \rightharpoonup \rho\quad \mbox{ in $L^\infty((0,T);L^2(\R^{d}))$-weak $\star$} $$
for all $T>0$ and we can write
$$f_\eps = \rho_\eps F  + g_\eps$$ 
where
 $g_\eps$ satisfies 
\begin{equation}\label{eq:gbd}
\| g_\eps\|_{ L^2 ( ( 0, T), L^2_{ M^{ -1}} ( \R^{ 2d}))} \leq C(T) \eps^{1/2}.
\end{equation}

This implies in particular that 
$$ f(x,v,t)=\rho(x,t) F(x,v,t).$$

In order to complete the proof of Theorem \ref{thm:2}, we need to pass to the limit in the weak formulation \eqref{weakform}.
First, we note that thanks to Proposition \ref{prop:2} (which we proved without restriction on $\alpha$), the right hand side in \eqref{weakform3} vanishes in the limit. Now let us define the 
operator $\L^\eps ( \varphi)$ as 

\begin{align}
  \L^\eps ( \varphi) &= \frac{ 1}{ \eps} \int_{\R^d} \nu ( v) F ( v, E) ( \chi_\eps - \varphi) \d v \nonumber \\
                     &= \frac{ 1}{ \eps} \int_{\R^d} \nu ( v) F ( v, 0) ( \chi_\eps - \varphi) \d v \nonumber \\
                     &\quad + \frac{ 1}{ \eps} \int_{\R^d} \nu ( v) \big( F( v, E) - F ( v, 0) \big) ( \chi_\eps - \varphi) \d v \nonumber \\
                     &= \L^\eps_1 ( \varphi) + \L^\eps_2 ( \varphi),
\end{align}
where $F ( v, 0) = M( v)$ thanks to the definition of $F$ given in \eqref{eq:kernelF}. 

Proposition 
4.4 in \cite{MR2815035} gives
\[
 \L^\eps_1 ( \varphi) \to \kappa ( -\Delta)^{ 1 / 2} \varphi \mbox{ in $L^2$-strong} .
\]
Furthermore, 
using formula \eqref{explicitaux} for $\chi_\eps$, 
we can recast $\L^\eps_2 ( \varphi)$ as follows:
\begin{align}
 \L^\eps_2 ( \varphi) & =\frac{ 1}{ \eps} \int_{\R^d} \int_0^\infty \e^{ - \nu ( v) z} \nu^2 ( v) \big( F( v, E) - M \big) ( \varphi ( x + \eps v z) - \varphi ( x)) \d z \d v\\ 
 & = \bigg(  \int_{\R^d}\int_0^\infty \e^{ - \nu ( v) z} \nu^2 ( v) \big( F( v, E) - M \big) v z \d z \d v \bigg) \cdot \nabla_x \varphi ( x, t) \nonumber \\
                      & \quad + \frac{ 1}{ \eps}\int_{\R^d}   \int_0^\infty \e^{ - \nu ( v) z} \nu^2 ( v) \big( F( v, E) - M \big) ( \varphi ( x + \eps v z) - \varphi ( x) - \eps v z \cdot \nabla \varphi ( x)) \d z \,  \d v, \nonumber \\
                      & = \mu ( E) \cdot \nabla_x \varphi ( x, t) + \mathcal{R}_\eps \nonumber 
\end{align}
and we can now show that $\mathcal{R}_\eps \to 0$ uniformly in $x$ and $t$: Indeed, Lemma \ref{lem:mu} implies
\begin{align*}
  | \mathcal{R}_\eps  | &\leq C \frac{ 1}{ \eps}\int_{\R^d}   \int_0^\infty \e^{ - \nu ( v) z} \nu^2 ( v) \frac{M(v)}{1+|v|} ( \varphi ( x + \eps v z) - \varphi ( x) - \eps v z \cdot \nabla \varphi ( x)) \d z \,  \d v 
\end{align*}
and for any $\eta \in [1,2]$, we have 
$$ |\varphi ( x + \eps v z) - \varphi ( x) - \eps v z \cdot \nabla \varphi ( x)|\leq C_\eta (\eps |v| z)^\eta .$$
We deduce
\begin{align*}
   |\mathcal{R}_\eps |  &\leq C_\eta  \eps^{\eta-1} \int_{\R^d}   \int_0^\infty \e^{ - \nu ( v) z} \nu^2 ( v) \frac{M(v)}{1+|v|} (|v| z)^\eta \d z \,  \d v 
\end{align*}
The integral in the right hand side is finite as long as  $\eta<2$ so we can take $\eta=3/2$ and deduce
$$ \|\mathcal{R}_\eps \|_{L^\infty}   \to 0 \mbox{ as } \eps\to 0$$

We have thus shown that
\[
 \L^\eps_2 ( \varphi) \to  \mu ( E) \cdot \nabla_x \varphi ( x, t)
\]
uniformly in $x$ and $t$ as $\eps \to 0$, which implies that  $\L^\eps ( \varphi)$ converges uniformly to
\[
 - \kappa ( - \Delta)^{ 1/2} ( \varphi) ( x, t) + \mu ( E) \cdot \nabla_x \varphi ( x, t).
\]
Passing to the limit in \eqref{weakform3} (the first two terms are handled exactly as in the proof of Theorem \ref{thm:main}), we deduce
$$
 \int_0^\infty \int \rho \dt \phi  \d x\d t + \int \rho^{ in} \phi |_{ t=0} \d x 
 + 
 \int_0^\infty \int \rho \left[ - \kappa ( - \Delta)^{ \alpha / 2} ( \phi) 
- \mu( E) \cdot \nabla_x \phi\right] \d x \d t = 0
$$
which is the weak formulation of \eqref{drifdifeq2}.

\end{proof}

\medskip

\begin{proof}[Proof of Lemma \ref{lem:mu}]
First, we note that for any $E\in\R^d$, the function $v\mapsto R(v,E)$ solves
$$ \T (R) = -   E \cdot \nabla_v M.$$
Using the coercivity property of $\T$ \eqref{coerineq} and the fact that $\int_{\R^d} R(v,E)\, \d v = 0$, we deduce
$$
 \left( \int \frac{ R^2}{ F} \d v \right)^{ 1 /2} \leq C |E|  \left( \int \frac{ |\nabla M|^2}{ F} \d v \right)^{ 1 /2} \leq C|E|.
 $$

Next, we rewrite the equation for $R$ as
\begin{equation}\label{eq:basref}
\nu R -  E \cdot \nabla_v R  = \K ( R) - E \cdot \nabla_v M.
\end{equation}
Using the fact that $\int_{\R^d} R(v,E)\, dv=0$, we can write
$$ \K(R)(v)  = \int_{\R^d} (\sigma(v,v')-\nu_0) R(v')\, dv',$$
and so using \eqref{eq:silim}, we obtain
\begin{align}
 |\K ( R)| & \leq \int | \sigma - \nu_0| | R ( v',E)| \d v' M( v) \nonumber \\
            & \leq \frac{ C M( v)}{ 1 + | v|} \int | R ( v',E)| \d v' \nonumber \\
           & \leq \frac{ C M( v)}{ 1 + | v|} \left( \int | R ( v',E)|^2 \frac{ \d v'}{ F( v',E)} \right)^{ 1 /2}\nonumber\\
           & \leq C|E| \frac{  M( v)}{ 1 + | v|} . 
             \label{Kest}
\end{align}
Finally,  assumptions \eqref{derdecay2} yields
\[
 | E \cdot \nabla_v M ( v) | \leq C|E| \frac{ M ( v)}{ 1 + | v|}.
\]
We thus have
$$|\nu R -  E \cdot \nabla_v R  | \leq C|E| \frac{ M ( v)}{ 1 + | v|},$$
which implies (using Remark \ref{frem})
$$ |R(v,E)| \leq  C|E| \frac{ M ( v)}{ 1 + | v|} \, \quad \mbox{ for all } v\in \R^d, \; E\in\R^d$$
and the lemma follows.
\end{proof}

\bibliographystyle{siam}

\bibliography{bibliography.bib}

\end{document}